\documentclass[12pt, reqno]{amsart}
\usepackage{mathrsfs}
\usepackage{amsfonts}
\usepackage{}
\usepackage{amsmath}
\usepackage{amsthm}
\usepackage{amssymb}
\usepackage[colorlinks=true, allcolors=blue]{hyperref}
\usepackage[compress,numbers]{natbib}
\usepackage{graphicx}
\usepackage{comment}
\usepackage{diagbox}
\usepackage{booktabs}
\usepackage{makecell}
\usepackage{tikz}
\newtheorem{theorem}{Theorem}[section]
\newtheorem{lemma}[theorem]{Lemma}
\usepackage{hyperref}
\usepackage{float}
\usepackage{multirow}

\usetikzlibrary{shapes, positioning}
\theoremstyle{definition}
\newtheorem{definition}[theorem]{Definition}

\newtheorem{proposition}[theorem]{Proposition}

\theoremstyle{remark}
\newtheorem{remark}[theorem]{Remark}

\numberwithin{equation}{section}

\newcommand{\abs}[1]{\lvert#1\rvert}

\newcommand{\R}{\mathbb{R}}
\newcommand{\Sph}{S^{n-1}}
\newcommand{\vol}{V_n}
\newcommand{\supp}{\operatorname{supp}}

\newcommand{\Span}{\operatorname{span}}
\newcommand{\Gr}{\operatorname{Gr}}
\newcommand{\KG}{\mathcal K_{G}}
\newcommand{\ip}[2]{\left\langle #1,#2\right\rangle}
\newcommand{\conv}{\operatorname{conv}}

\newcommand{\Rn}{\mathbb{R}^{n}}
\newcommand{\Sn}{S^{n-1}}
\newcommand{\On}{\mathrm{O}(n)}

\newcommand{\SL}{S^{n-1}\cap L}
\newcommand{\Lp}{L_{p}}

\newcommand{\la}{\langle}
\newcommand{\ra}{\rangle}
\newcommand{\lp}{L^\perp}

\newcommand{\Fix}{\mathrm{Fix}}
\DeclareMathOperator{\diam}{diam}

\theoremstyle{definition}

\begin{document}
	
	\title{The Dual Minkowski Problem under Group Actions}
	\author{Junjie Shan}
	\address{School of Mathematics, Sichuan University, Chengdu, Sichuan, 610064, P. R. China}

\email{shanjjmath@163.com 
}

\thanks{{\it 2010 Math Subject Classification}:  52A40, 52A38}

\thanks{{\it Keywords}:    dual Minkowski problem,  dual curvature measure, group symmetry, $G$-invariant convex bodies, $G$-invariant subspace, subspace concentration condition}

\begin{abstract}
In this paper, we study the dual Minkowski problem under group symmetry. For $0<q\le n$, we give a complete existence characterization in the framework of $G$-invariant convex bodies when the group $G\subset O(n)$ has no nonzero fixed points, recovering the origin-symmetric setting when $G=\{\pm I\}$. The necessary and sufficient conditions concern the concentration of the measure on $G$-invariant subspaces, both in the range $0<q<n$ and at the critical endpoint $q=n$, where the problem becomes the logarithmic Minkowski problem.
\end{abstract}
	
	\maketitle

\section{Introduction}
A central problem in convex geometry is the study of global geometric invariants and locally defined geometric measures of convex bodies. The classical Minkowski problem concerns the characterization of the surface area measure, which may be viewed as the differential of the volume functional. It asks for necessary and sufficient conditions under which a given Borel measure on the unit sphere is the surface area measure of a convex body. The solution of the classical Minkowski problem has found numerous applications in various areas of analysis and geometry.

During the last three decades, many new geometric measures have been introduced and extensively studied. 
	Lutwak \cite{Lutwak1993,Lutwak1996} introduced the $L_p$ Brunn-Minkowski theory and posed the $L_p$ Minkowski problem of characterizing $L_p$ surface area measures. The $L_p$ Minkowski problem has been solved for $p \ge 1$; see \cite{lyz2004,Lutwak1993,L-O 1995}. For $p<1$, however, the $L_p$ Minkowski problem remains largely open; for some progress, see \cite{B-B-A-Y 2019,chou-wang,GLW,JJZ2016,LW2013,zhu2015,zhu2017}. A particularly important unsolved case is the celebrated logarithmic Minkowski problem corresponding to $p=0$. Unlike the classical Minkowski problem, the logarithmic Minkowski problem exhibits a measure concentration phenomenon.  B\"or\"oczky, Lutwak, Yang, and Zhang~\cite{KLYZ2013JAMS} gave a complete solution in all dimensions in the even case. They showed that the subspace concentration condition is the necessary and sufficient condition for the existence of a solution.
	
	A nonzero finite Borel measure $\mu$ on $S^{n-1}$ is said to satisfy the \textit{subspace concentration condition} if
	\begin{equation}\label{subspace concentration condition}
		\frac{\mu(S^{n-1}\cap \xi)}{\mu(S^{n-1})}\le \frac{\dim \xi}{n}
	\end{equation}
	for every  subspace $\xi$ of $\mathbb R^n$ with $0<\dim \xi<n$, and whenever equality holds for some $\xi$, there exists a subspace $\xi'$ complementary to $\xi$ such that $\mu$ is concentrated on $S^{n-1}\cap (\xi\cup \xi')$.

The logarithmic Minkowski problem in the non-even case remains largely open. Partial results concerning polytopes and approximation by polytopes were obtained in
\cite{BHZ2016,CLZ2019,HLL2006,HL2014,Stancu 2002,Xiong 2010,zhu2014}. B\"or\"oczky and Henk \cite{BH2016} showed that the cone-volume measure of a convex body with centroid at the origin still satisfies the subspace concentration condition.

In the groundbreaking work \cite{HLYZdual}, Huang, Lutwak, Yang, and Zhang introduced a new family of geometric measures, called dual curvature measures. For a convex body $K$ and $q\in\mathbb R$, the $q$-th dual curvature measure is denoted by $\widetilde C_q(K,\cdot)$ and can be viewed as a differential of the dual quermassintegral. These measures connect two well-known fundamental geometric measures. The case $q=0$ corresponds, after a suitable normalization, to Aleksandrov's integral curvature measure, while the case $q=n$ gives the cone-volume measure  mentioned above.
The dual Minkowski problem, which asks for the characterization of dual curvature measures, remains widely open:

\textit{Given a nonzero finite Borel measure $\mu$ on the unit sphere $S^{n-1}$ and a real number $q$, what are necessary and sufficient conditions for the existence of a convex body $K$ satisfying
	$
	\mu=\widetilde C_q(K,\cdot)?
	$}
	
	Thus the Aleksandrov problem and the logarithmic Minkowski problem arise as two critical cases of the dual Minkowski problem. For $0<q<n$, the dual Minkowski problem for even data has been completely solved. The existence of solutions was first proved by Huang, Lutwak, Yang, and Zhang \cite{HLYZdual} for $0<q\le 1$, and for $1<q<n$ under a stronger sufficient condition. This condition was later refined by Zhao \cite{zhao2018} for integer $q$, and independently by B\"or\"oczky, Henk, and Pollehn \cite{BHP2018} for all $1<q<n$. Zhao~\cite{zhao2018} also proved sufficiency for integer $q$, and the sufficiency for all real $1<q<n$ was later established by B\"or\"oczky, Lutwak, Yang, Zhang, and Zhao~\cite{BLYZZ2019}. They showed that the $q$-th subspace mass inequality is the necessary and sufficient condition for the existence of a solution.
	
	A nonzero finite Borel measure $\mu$ on $S^{n-1}$ is said to satisfy the \textit{$q$-th subspace mass inequality} if
\begin{equation}\label{qthclassical}
\frac{\mu( S^{n-1}\cap \xi)}{\mu(S^{n-1})}
<
\min\left\{\frac{\dim \xi}{q},1\right\}
\end{equation}
for every subspace $\xi$ of $\mathbb R^n$ with $0<\dim \xi<n$.

When $0<q\le 1$, this condition simply means that $\mu$ is not concentrated on any great subsphere. For further progress on the dual Minkowski problem, see \cite{BLYZ2015,chenli 2018,CLWX2025,EH2023,guoxizhao,HLYZ2018,HLXZ,LYZX chord,lsw,Lyzdualcone,zhao2017}. However, little is known about the dual Minkowski problem in the non-even case.
 
 In \cite{lyz2018}, the $(p,q)$-th dual curvature measure was introduced, unifying the $L_p$ surface area measure and the dual curvature measure mentioned above. For further developments, see \cite{BF2019,chenli flow,GLW2023 flow,huangzhao dual,LLL2022,mui2024}. If the prescribed measure $\mu$ has a density $f$ on $S^{n-1}$, then no subspace concentration
 phenomenon occurs. In this absolutely continuous setting, the $L_p$ dual Minkowski problem reduces to the following Monge-Amp\`ere type equation:
 \begin{equation}\label{continuous}
 	\det(\nabla_{ij} h+h\delta_{ij})
 	=
 	h^{p-1}\big(|\nabla h|^2+h^2\big)^{\frac{n-q}{2}}f,
 \end{equation}
 where $h$ is the unknown support function, $\delta_{ij}$ is the Kronecker delta, and $\nabla h$ and $\nabla_{ij}h$ denote the gradient and Hessian of $h$ on $S^{n-1}$, respectively.

We first recall some notation. Let $\mathcal K^n$ denote the set of convex bodies in $\mathbb R^n$, let $\mathcal K_o^n$ denote the set of convex bodies containing the origin in their interior, and let $\mathcal K_e^n$ denote the set of origin-symmetric convex bodies.

The aim of this paper is to study the existence part of the dual Minkowski problem in the framework of $G$-invariant convex bodies.
Let $G \subset O(n)$ be a subgroup. A convex body $K \subset \mathbb R^n$ is called $G$-invariant if $gK=K$ for all $g \in G$. We denote by $\mathcal{K}_{G}$ the collection of all \textit{$G$-invariant convex bodies}, i.e.,
\begin{equation*}
	\mathcal{K}_{G} := \left\{ K \in \mathcal{K}^n : gK = K \text{ for all } g \in G \right\}.
\end{equation*}
A Borel measure $\mu$ on $S^{n-1}$ is \textit{$G$-invariant} if $
\mu(gE) = \mu(E)$ for all $g \in G$ and every Borel set  $E \subset S^{n-1}$. From this algebraic viewpoint, origin-symmetric convex bodies can be regarded as $\{\pm I\}$-invariant convex bodies, where $I$ denotes the identity map; that is, $\mathcal K_{\{\pm I\}}=\mathcal K_e^n$.

Two classes of groups play a particularly important role in the study of $G$-invariant Minkowski-type problems. The first class consists of irreducible groups. We say that a subgroup $G\subset O(n)$ is irreducible if the only $G$-invariant subspaces of $\mathbb R^n$ are $\{0\}$ and $\mathbb R^n$ (that is, no nontrivial  subspace $V\subset \mathbb R^n$ satisfies $gV=V$ for all $g\in G$). Such symmetry is highly rigid. In \cite{shan2026}, we showed that irreducibility is strong enough to ensure the existence of solutions to the $L_p$ dual Minkowski problem for all $p,q\in\mathbb R$. Note that a convex body invariant under an irreducible group need not be origin-symmetric; for example, one may take $G$ to be the symmetry group of a regular simplex. In particular, when $n=2$ and $G$ is the rotation symmetry group of an equilateral triangle, $G$-invariant convex bodies are precisely those invariant under rotation by $120^\circ$. In Section~3, we also prove that every finite Borel measure on $S^{n-1}$ invariant under an irreducible group automatically satisfies the subspace concentration condition \eqref{subspace concentration condition}. 

The second class consists of groups  with no nonzero fixed points. By \cite{shan2026}, $\mathcal K_G\subset \mathcal K_o^n$ if and only if $G$ has no nonzero fixed points. Moreover, every irreducible group has no nonzero fixed points when $n\ge 2$, and
the group $\{\pm I\}$ clearly has no nonzero fixed points. The condition of having no nonzero fixed points is substantially weaker than
irreducibility; several examples are given in Section~3.
Very recently, B\"or\"oczky, Kov\'acs, Mui, and Zhang~\cite{BKMZ group} studied the absolutely continuous equation \eqref{continuous} from the viewpoint of group symmetry and solved it for $q>0$ and for a certain range of negative $p$, under the assumptions that the density $f$ is
invariant under a closed subgroup with no nonzero fixed points and satisfies suitable regularity conditions.  The relations among these classes of convex bodies are as follows when $n\ge2$:

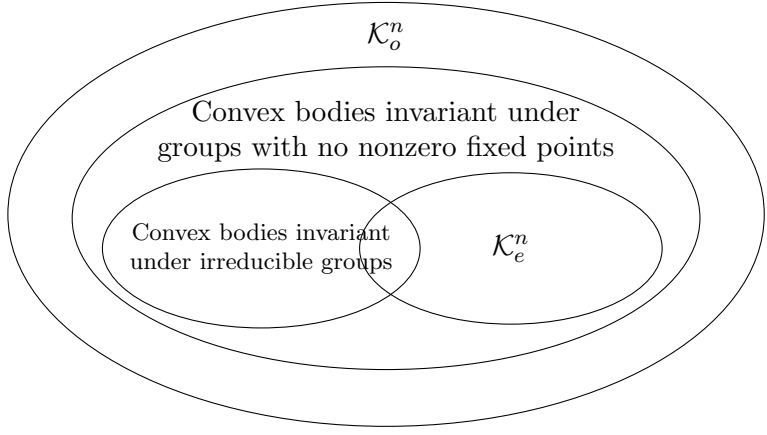
\begin{figure}[H]
	\centering
	\begin{tikzpicture}[font=\small]
		
		\draw (0,0) ellipse (5.0cm and 2.8cm);
		\node at (0,2.35) {$\mathcal K_o^n$};
		
		\draw (0,-0.05) ellipse (4.15cm and 2.0cm);
		\node[align=center] at (0,1.10)
		{Convex bodies invariant under\\
			groups with no nonzero fixed points};
		
		\draw (-1.65,-0.45) ellipse (2.1cm and 1.05cm);
		\node[align=center, font=\scriptsize] at (-1.65,-0.45)
		{Convex bodies invariant\\
			under irreducible groups};
		
		\draw (1.65,-0.45) ellipse (2.0cm and 1.0cm);
		\node at (1.65,-0.45) {$\mathcal K_e^n$};
		
	\end{tikzpicture}
\end{figure}

In this paper, we give a complete existence characterization for the dual Minkowski problem for $0<q\le n$ in the class of $G$-invariant convex bodies, where $G\subset O(n)$ has no nonzero fixed points. The obstruction is expressed in terms of the concentration of the measure on $G$-invariant subspaces. For $0<q<n$, this leads to the following condition.

\vspace{4pt}
\noindent\textbf{$G$-invariant $q$-th subspace mass inequality.}
We say that a nonzero  finite  Borel measure $\mu$ on $S^{n-1}$ satisfies the \textit{$G$-invariant $q$-th subspace mass inequality} if
\begin{equation}\label{G-invariant q-th subspace mass inequality}
	\frac{\mu(  S^{n-1}\cap L)}{\mu(S^{n-1})}
	<
	\min\left\{\frac{\dim  L}{q},1\right\}
\end{equation}
for every \textbf{$G$-invariant subspace} $ L$ of $\mathbb R^n$ with $0<\dim  L<n$.

The $G$-invariant $q$-th subspace mass inequality only controls the mass distribution on $G$-invariant subspaces. Note that there may be very few such subspaces. For example, if $G$ is irreducible, then there are no nontrivial $G$-invariant subspaces, so \eqref{G-invariant q-th subspace mass inequality} is vacuous in this case. Furthermore, let
$
\R^n=V_1\oplus\cdots\oplus V_m
$
be an orthogonal decomposition, and let
\begin{equation*}
	G=G_1\times\cdots\times G_m\subset O(V_1)\times\cdots\times O(V_m)\subset O(n),
\end{equation*}
 where each $G_j$ is irreducible on $V_j$ (when $\dim V_j=1$ we assume $G_j\neq \{I_{V_j}\}$). Then the $G$-invariant subspaces of $\mathbb R^n$ are precisely
$
\bigoplus_{j\in I}V_j
$
for some subset $I\subset\{1,\dots,m\}$. At the other extreme, when $G=\{\pm I\}$, every subspace of $\mathbb R^n$ is $\{\pm I\}$-invariant. Hence the $\{\pm I\}$-invariant $q$-th subspace mass inequality is exactly  \eqref{qthclassical}.

The existence of solutions relies on the geometric structure of $G$-invariant convex bodies. The sufficiency direction is proved by a variational method. Since only mass information on $G$-invariant subspaces is available,  we  develop a new technical framework for estimates adapted to  $G$-invariant subspace blocks. For $0<q<n$, this leads to the following characterization.

	\begin{theorem}\label{thm:main-q}
	Let $q \in (0,n)$, let $G \subset O(n)$ be a subgroup with no nonzero fixed points, and let $\mu$ be a nonzero finite $G$-invariant Borel measure on $S^{n-1}$. Then there exists a $G$-invariant convex body $K$ in $\mathbb R^n$ such that
	\[
\mu=\widetilde	C_q(K,\cdot)
	\]
	if and only if $\mu$ satisfies the $G$-invariant $q$-th subspace mass inequality.
\end{theorem}

In the critical case $q=n$, which corresponds to the logarithmic Minkowski problem, the measure concentration phenomenon takes a different form.

\vspace{1.7pt}
\noindent\textbf{$G$-invariant subspace concentration condition.}
We say that a nonzero finite Borel measure $\mu$ on $S^{n-1}$ satisfies the \textit{$G$-invariant subspace concentration condition} if
\begin{equation}\label{$G$-invariant subspace concentration condition}
\frac{\mu(S^{n-1}\cap L)}{\mu(S^{n-1})}
\le \frac{\dim L}{n}
\end{equation}
	for every $G$-invariant subspace $L$ of $\mathbb R^n$ with $0<\dim L<n$, and whenever equality holds for some $G$-invariant subspace $L$, there exists a subspace $\widetilde L$ complementary to $L$ such that $\mu$ is concentrated on $S^{n-1}\cap (L\cup \widetilde L)$.

As in the case of \eqref{G-invariant q-th subspace mass inequality}, this condition only tests the mass distribution of $\mu$ on $G$-invariant subspaces. For the $G$-invariant logarithmic Minkowski problem, it provides the necessary and sufficient condition for the existence of a solution.

\begin{theorem}\label{thm:main-log}
	Let $G\subset O(n)$ be a subgroup with no nonzero fixed points. A nonzero finite $G$-invariant Borel measure $\mu$ on $S^{n-1}$ is the cone-volume measure of a $G$-invariant convex body in $\mathbb R^n$ if and only if $\mu$ satisfies the $G$-invariant subspace concentration condition.
\end{theorem}

When $G=\{\pm I\}$, Theorems~1.1 and~1.2 reduce to the corresponding results
in the even setting.


At first sight, the $G$-invariant subspace concentration condition is much weaker than the classical subspace concentration condition.
The fewer $G$-invariant subspaces there are, the less information this condition contains; in a certain sense, however, this reflects the stronger symmetry imposed by $G$. We obtain solutions to the $G$-invariant logarithmic Minkowski problem under this condition, which indicates that the group symmetry precisely compensates for the missing  information.

If $G$ is irreducible, then the $G$-invariant subspace concentration condition is vacuous; however, Proposition~\ref{irreducible-subspace-mass} shows that every nonzero finite Borel measure on $S^{n-1}$ invariant under an irreducible group automatically satisfies the classical subspace concentration condition. At the opposite extreme, when $G=\{\pm I\}$, every subspace of $\mathbb R^n$ is $G$-invariant, and hence the $G$-invariant subspace concentration condition coincides with the classical  subspace concentration condition. Thus, in both extreme cases, the condition imposed on $G$-invariant subspaces already yields the classical subspace concentration condition on all subspaces.

It is therefore natural to ask whether this implication holds beyond these two extreme cases. Combining Theorem~\ref{thm:main-log} with the result in \cite{BH2016} on cone-volume measures of centered convex bodies, we obtain the following structural consequence for $G$-invariant measures. It shows that the $G$-invariant subspace concentration condition already
controls concentration on all subspaces.

\begin{theorem}\label{thm equivalence}
	Let $G\subset O(n)$ be a subgroup with no nonzero fixed points. Let $\mu$ be a nonzero finite $G$-invariant Borel measure on $\Sph$. Then the following are equivalent:
	
	\smallskip
	
	\noindent
	\textup{(i)} The measure $\mu$ satisfies the $G$-invariant subspace concentration condition; namely, for every $G$-invariant subspace $L\subset \R^n$,
	\[
	\frac{\mu( \Sph\cap L)}{\mu(\Sph)}
	\le \frac{\dim L}{n},
	\]
	and whenever equality holds for some $G$-invariant subspace $L\subset \R^n$, there exists a subspace $\widetilde L$ complementary to $L$ in $\R^n$ such that $\mu$ is concentrated on $S^{n-1}\cap (L\cup \widetilde L)$.
	
	\smallskip
	
	\noindent
	\textup{(ii)} The measure $\mu$ satisfies the classical subspace concentration condition; namely, for every subspace $\xi\subset \R^n$,
	\[
	\frac{\mu( \Sph \cap \xi)}{\mu(\Sph)}
	\le \frac{\dim \xi}{n},
	\]
	and whenever equality holds for some subspace $\xi\subset \R^n$, there exists a subspace $\xi'$ complementary to $\xi$ in $\R^n$ such that $\mu$ is concentrated on $S^{n-1}\cap (\xi\cup \xi')$.
\end{theorem}

This reveals a surprising geometric structure of measures under group symmetry,
and the Minkowski problem serves as a bridge between algebra and geometry. We
note that the bound $\dim(\cdot)/n$ in Theorem~\ref{thm equivalence} is sharp.
Indeed, if this bound were replaced by a strictly smaller one, then the equivalence
in Theorem~\ref{thm equivalence} would fail in general; this can be seen by
taking
$G$ to be a suitable irreducible group, such as the rotation symmetry group
of a cube.

	\section{Preliminaries}
	Let $\left( \mathbb{R}^n, \langle \cdot, \cdot \rangle\right) $ denote the $n$-dimensional Euclidean space  with the standard inner product. 
	
	Let $S^{n-1}$ denote the unit sphere and $B^{n}$ the closed unit ball centered at the origin in $\mathbb{R}^n$. Let $\mathcal{B}^n$ denote the class of Euclidean balls centered at the origin in $\mathbb{R}^n$.
	
		For a linear subspace $U\subset \R^n$, we write $P_U$ for the orthogonal projection onto $U$,  $B_U$ for the Euclidean unit ball of $U$, and $A|U=P_UA$ for the orthogonal projection of a set $A\subset\R^n$ onto $U$.  If $U$ has dimension $k$, then $V_k(\cdot)$ denotes $k$-dimensional Lebesgue measure on $U$.
		We say that a subspace $U$ of $\R^n$ is nontrivial (or proper) if $U\neq \{0\}$ and $U\neq \R^n$.
	
	A \textit{convex body} is a compact convex subset of $ \mathbb{R}^{n} $ with nonempty interior. The set of convex bodies in  $\mathbb{R}^{n} $ is denoted by $ \mathcal{K}^{n}$,  the set of origin-symmetric convex bodies in  $\mathbb{R}^{n} $ is denoted by $ \mathcal{K}_{e}^{n}$, and $ \mathcal{K}_{o}^{n} $ denotes the set of convex bodies in $\mathbb{R}^{n}$ that contain the origin in their interior.
	
	The \textit{support function} of a compact convex set $ K $ is given by 
	$$ h_{K}(x)=\max _{y\in K} \la y , x \ra, \quad \text{for $x\in \Rn$}.$$
	
	A compact set $K$ is \textit{star-shaped} with respect to a point $x$ if every line through $x$ that meets $K$ does so in a (possibly degenerate) closed line segment.
	If $K$ is a  star-shaped set with respect to the origin, then for all $u\in S^{n-1}$, its \textit{radial  function} $ \rho_K(\cdot )$   is defined by (see \cite{Gaedner gtbook,schneiderbook2014})
	\[
	\rho_K(u)=\max\{c:cu\in K\}.
	\]
	The class of \textit{star bodies} $\mathcal{S}^{n}_{o}$ in $\Rn$ is defined to be the family of star-shaped sets whose radial functions are positive and continuous. The radial function is  homogeneous of degree $-1$, that is,
	\begin{align}\label{rad_fun_homo}
		\rho_K(cu)=c^{-1}\rho_K(u), \quad c>0.
	\end{align}
	For an invertible linear transformation $\phi$, the support function $h_K$ and the radial function $\rho_K$ transform as follows:
	\begin{align}
		h_{\phi K}(y) &= h_K(\phi^T y), \label{eq:support-transform} \\
		\rho_{\phi K}(y) &= \rho_K(\phi^{-1} y). \label{eq:radial-transform}
	\end{align}
	If $K\in \mathcal{S}_o^k$ and $L\in \mathcal{S}_o^{n-k}$, then, for $(x,y)\in \mathbb{R}^k\times \mathbb{R}^{n-k}$ with $x\neq 0$ and $y\neq 0$,
	\begin{equation}\label{product radial}
		\rho_{K\times L}(x,y)=\min\{\rho_K(x),\rho_L(y)\}.
	\end{equation}

	
	For a convex body $K \in \mathcal{K}_o^n$ and a Borel set $\eta \subset S^{n-1}$, define
	\[
	\alpha_K^*(\eta) =  \big\{u \in S^{n-1} : \rho_K(u)u \in \nu_K^{-1}(\eta)\big\},
	\]
	where  $\nu_{K}$ is the outer unit normal vector  of $K$.
	
Dual quermassintegrals are fundamental geometric invariants in the dual Brunn-Minkowski theory. For each $K\in \mathcal{S}^{n}_{o}$ and $q=1,\ldots,n$,
$$
\widetilde{W}_{n-q}(K)=\frac{\omega_n}{\omega_q}\int_{\Gr(n,q)} \mathcal{H}^q(K\cap\xi)\,d\xi,
$$
where the integral is taken with respect to the rotation-invariant probability measure on the Grassmannian $\Gr(n,q)$ of all $q$-dimensional subspaces of $\mathbb{R}^n$. Here $\mathcal{H}^q$ denotes the $q$-dimensional Hausdorff measure, and $\omega_q$ denotes the $q$-dimensional volume of the unit ball in $\mathbb{R}^q$.	
	For general $q\in \mathbb{R}$, the $(n-q)$-th dual quermassintegral of $K$ is defined by
	$$
	\widetilde W_{n-q}(K)=\frac{1}{n}\int_{\Sph}\rho_K(u)^q\,du.
	$$

	For $q\in\mathbb R$, the $q$-th dual curvature measure of a convex body $K\in\mathcal K_o^n$ is defined by
	\begin{equation}\label{qth-dual-curvature-measure}
		\widetilde C_q(K,\eta)
		=\frac{1}{n}\int_{\alpha_K^*(\eta)} \rho_K(u)^q\,d\mathcal H^{n-1}(u),
	\end{equation}
	for every Borel set $\eta\subset S^{n-1}$.
	In particular, when $q=n$, $\widetilde C_n(K,\cdot)$ is the cone-volume measure $V_K$, namely
	\begin{equation}\label{cone-volume-measure}
		\widetilde C_n(K,\eta)=V_K(\eta)
		=\frac{1}{n}\int_{\nu_K^{-1}(\eta)}
		\langle x,\nu_K(x)\rangle\,d\mathcal H^{n-1}(x).
	\end{equation}
By \cite{BHP2018}, the dual curvature measure admits the following integral
representation: for each $K\in\mathcal K_o^n$, Borel set
$\eta\subset S^{n-1}$, and $q>0$,
\begin{equation}\label{Cq representation}
	\widetilde C_q(K,\eta)
	=
	\frac{q}{n}
	\int_{\substack{x\in K\setminus\{0\}\\ x/|x|\in \alpha_K^*(\eta)}}
	|x|^{\,q-n}\,d\mathcal H^n(x).
\end{equation}
	
	The set of continuous positive functions on the sphere $ S^{n-1} $ will be denoted by $ C^{+}(S^{n-1}) $. For each $ f\in C^{+}(S^{n-1}) $, the Wulff shape $ [f] $ generated by $ f $ is a convex body defined by
	\begin{equation}\label{Wulff shape}
		[f]=\left\{x\in \mathbb{R}^{n}: \la x, v \ra \le f(v)\; \text{for all}\; v\in S^{n-1}\right\}.
	\end{equation}

The variational formula for dual quermassintegrals was established in \cite{HLYZdual,lyz2018}. More precisely, for $q\neq 0$ and $f\in C(S^{n-1})$, the $q$-th dual curvature measure $\widetilde{C}_q(K,\cdot)$ is characterized by the logarithmic Wulff perturbation formula
\begin{equation}\label{vari}
	\left.\frac{d}{dt}\widetilde W_{n-q}([h_t])\right|_{t=0}
	=
	q\int_{\Sph} f(v)\,d\widetilde{C}_q([h_0],v),
\end{equation}
whenever $h_0\in C^+(\Sph)$ and
$
\log h_t=\log h_0+t f+o(t)
$
in $C(\Sph)$ for sufficiently small $|t|$.

We shall use the following Gaussian integral representation.

\begin{lemma}[Lemma 5.2, \cite{BLYZZ2019}]\label{lem:gauss}
	Let $0<q<n$, and let $S$ be a star body in $\R^n$. Then
	\[
	\int_{\R^n}\rho_S(z)^q e^{-\abs{z}^2}\,dz
	=
	c_0(n,q)\,\widetilde W_{n-q}(S),
	\]
	where
	\[
	c_0(n,q)=n\int_0^\infty e^{-r^2}r^{n-q-1}\,dr.
	\]
\end{lemma}

We shall also use the following elementary weighted-sum lemma.

\begin{lemma}[Lemma 4.1, \cite{BLYZZ2019}]\label{lem:weights}
	Let $\lambda_1,\dots,\lambda_m\in [0,1]$ satisfy
	\[
	\lambda_1+\cdots+\lambda_m=1.
	\]
	Let $a_1\le a_2\le \cdots\le a_m$ be real numbers. Suppose that 
	$\sigma_0,\sigma_1,\dots,\sigma_m\in[0,\infty)$ satisfy
	\[
	\sigma_0=0,\qquad \sigma_m=1,\qquad
	\lambda_1+\cdots+\lambda_i\le \sigma_i
	\quad (1\le i\le m-1).
	\]
	Then
	\[
	\sum_{i=1}^m \lambda_i a_i
	\ge
	\sum_{i=1}^m (\sigma_i-\sigma_{i-1})a_i.
	\]
\end{lemma}

	For a subgroup $G$ of $O(n)$,  a convex body $K \subset \mathbb{R}^n$ is \textit{$G$-invariant} if it satisfies
	\[
	gK= K \quad \text{for all $g \in G$}.
	\]
	A function $f$ defined on $\Sn$ is \textit{$G$-invariant} if
	\[
	f(gx) = f(x) \quad \text{for all $ g \in G, x\in \Sn$}.
	\]
	A Borel measure $\mu$ on $S^{n-1}$ is \textit{$G$-invariant} if
	\[
	\mu(gE) = \mu(E) \quad \text{for all $g \in G$  and all Borel sets  $E \subset S^{n-1}$}.
	\]
	A subspace $L\subset\mathbb R^n$ is \textit{$G$-invariant} if
	\[
	gL=L \quad \text{for all } g\in G.
	\]

	 We write
	\[
	\KG=\{K\in \mathcal K^n:\ gK=K\text{ for all }g\in G\}.
	\]
	 Note that, if $f\in C^{+}(\Sn)$ is a $G$-invariant function, then the Wulff shape $[f]$ is clearly a $G$-invariant convex body. Since for any $g\in G$,
	\begin{align*}
		g[f]=&\left\{gx\in \mathbb{R}^{n}: \la x, u \ra \le f(u),\;  u\in S^{n-1}\right\}\\
		=&\left\{y\in \mathbb{R}^{n}: \la g^{-1}y, u \ra \le f(u),\;  u\in S^{n-1}\right\}\\
		=&\left\{y\in \mathbb{R}^{n}: \la y, gu \ra \le f(gu),\;  u\in S^{n-1}\right\}\\
		=&\left\{y\in \mathbb{R}^{n}: \la y, v \ra \le f(v),\;  v\in S^{n-1}\right\}=[f].
	\end{align*}
	Conversely, if $K$ is a $G$-invariant convex body, then its support function is naturally $G$-invariant from \eqref{eq:support-transform}. 
	 
	 A complete classification of $G$-invariant convex bodies in relation to several classical classes of convex bodies is summarized in the following table; see \cite{shan2026}.
	  \begin{table}[h]
	 	\centering
	 	\begin{tabular}{|c|c|c|c|c|}
	 		\hline
	 		\diagbox[width=5.5em, height=1.8em]{}{$\mathcal{C}$} & $\mathcal{B}^n$ & $\mathcal{K}_e^n$ & $\mathcal{K}_o^n$ & $\mathcal{K}^n$ \\ 
	 		\hline
	 		$\mathcal{K}_{G} = \mathcal{C}  \iff$ & 
	 		\begin{minipage}{0.18\textwidth}
	 			\vspace{0.2em}
	 			$\overline{Gv}=\Sn$ for all  $v\in S^{n-1}$
	 			\vspace{0.2em}
	 		\end{minipage} & 
	 		\begin{minipage}{0.18\textwidth}
	 			$G = \{\pm I\}$
	 		\end{minipage} & 
	 		\begin{minipage}{0.18\textwidth}
	 			No such $G$
	 		\end{minipage} &
	 		\begin{minipage}{0.18\textwidth}
	 			$G=\{I\}$
	 		\end{minipage} \\ 
	 		\hline
	 		$\mathcal{K}_{G} \subset \mathcal{C} \iff$ & 
	 		\begin{minipage}{0.18\textwidth}
	 			\vspace{0.2em}
	 			$\overline{Gv}=\Sn$ for all  $v\in S^{n-1}$
	 			\vspace{0.2em}
	 		\end{minipage} & 
	 		\begin{minipage}{0.18\textwidth}
	 			\vspace{0.2em}
	 			$-x \in \overline{Gx}$ for all $x\in \Rn$
	 			\vspace{0.2em}
	 		\end{minipage} & 
	 		\begin{minipage}{0.18\textwidth}
	 			\vspace{0.2em}
	 			No nonzero fixed points
	 			\vspace{0.2em}
	 		\end{minipage} &
	 		\begin{minipage}{0.18\textwidth}
	 			Always true
	 		\end{minipage} \\ 
	 		\hline
	 		$\mathcal{K}_{G} \supset \mathcal{C} \iff$ & 
	 		\begin{minipage}{0.18\textwidth}
	 			\vspace{0.2em}
	 			Always true
	 			\vspace{0.2em}
	 		\end{minipage} & 
	 		\begin{minipage}{0.18\textwidth}
	 			\vspace{0.2em}
	 			$G = \{I\}$ or $\{\pm I\}$
	 			\vspace{0.2em}
	 		\end{minipage} & 
	 		\begin{minipage}{0.18\textwidth}
	 			$G = \{I\}$
	 		\end{minipage} &
	 		\begin{minipage}{0.18\textwidth}
	 			$G=\{I\}$
	 		\end{minipage} \\
	 		\hline
	 	\end{tabular}
	 \end{table} 
	 
	 \noindent  where $Gv = \{ gv : g \in G \}$ is the orbit and $\overline{Gv}$ is its closure.

	The class of $G$-invariant convex bodies is  closed under the operations of   $L_p$ Minkowski addition $+_{p}$, scalar multiplication, polar  operation $*$, and convex hull operation  (if $K, L\in \mathcal{K}_{G}$, then $\conv(K, L) \in \mathcal{K}_{G}$). The algebraic structure of group-invariant convex bodies ensures their abundance.
	
	We say a group $G$ is irreducible if its action on $\mathbb{R}^n$ is irreducible  (i.e., the only subspaces $V \subset \mathbb{R}^n$ satisfying $gV = V$ for all $g \in G$ are $V = \{0\}$ and $V = \mathbb{R}^n$).  Let
	\[
	\mathrm{Fix}(G):=\{x\in \mathbb R^n: gx=x \text{ for all } g\in G\}
	\]
	denote the set of fixed points of the group $G$.
	If $G\subset O(n)$ is irreducible, it has no nonzero fixed points ($n\ge2$). If there exists a nonzero fixed point $x_1 \in \mathbb{R}^n$ satisfying $gx_1 = x_1$ for all $g \in G$, then the line $\mathbb{R}x_1$ forms a nontrivial $G$-invariant subspace, contradicting the irreducibility. 
	
	Since the centroid of a convex body is equivariant under linear transformations, if $G\subset O(n)$ satisfies $\mathrm{Fix}(G)=\{0\}$, then the centroid of every $G$-invariant convex body lies at the origin.

	For $1\le k\le n$, we denote by $\Gr(n,k)$ the Grassmannian of all $k$-dimensional linear subspaces of $\mathbb R^n$. For $E,F\in \Gr(n,k)$, let $P_E$ and $P_F$ be the orthogonal projections onto $E$ and $F$, respectively, and define
	\begin{equation}\label{grassmaiann metric}
	d(E,F):=\|P_E-P_F\|,
	\end{equation}
	where, for a linear operator $T:\mathbb R^n\to\mathbb R^n$,
	\begin{equation}\label{normdef}
	\|T\|:=\sup\{|Tx|:x\in \mathbb R^n,\ |x|=1\}
	\end{equation}
	denotes the operator norm induced by the Euclidean norm. This makes $\Gr(n,k)$ into a compact metric space.
	

\section{Measure concentration on $G$-invariant subspaces}
	
If $G\subset O(n)$ has no nonzero fixed points and $K$ is a $G$-invariant convex body, then the centroid of $K$ is fixed by every element of $G$. Hence the centroid of $K$ is the origin.
To estimate the dual curvature measures of $G$-invariant convex bodies, we need the following lemma.
	\begin{lemma}\label{lem:parallel-sections}
Let $G\subset O(n)$,	 and let $E$ be a proper $G$-invariant subspace of $\mathbb R^n$ with $\dim E=m\ge1$. Let $a\in E^\perp$, let $C\subset a+E$ be a nonempty compact convex set,  let  $h_1,\dots,h_N\in G$, and set $h_0=I$. Put
		\[
		 C_i=h_iC\subset h_i a+E,\qquad i=0,\dots,N.
		\]
		If $\beta_0,\dots,\beta_N\ge 0$ and $\sum_{i=0}^N\beta_i=1$, then
		\[
		M:=\sum_{i=0}^N \beta_i C_i
		\]
		satisfies
		\begin{equation}\label{3.1}
		\int_M |x|^{\,q-n}\,d\mathcal H^m(x)\ge \int_C |x|^{\,q-n}\,d\mathcal H^m(x)
		\end{equation}
		for every $q\in(0,n)$. Moreover, if both integrals are finite, then equality holds if and only if
		\[
		\mathcal H^m(M\cap RB^n)=\mathcal H^m(C\cap RB^n)
		\]
		for every $R>0$.
	\end{lemma}
	
	\begin{proof}
		For $R>0$, define
	\begin{equation}\label{Di}
		D_i(R):=C_i\cap RB^n,\qquad i=0,\dots,N.
	\end{equation}
		Since both $C_i$ and $RB^n$ are convex, we have
		\begin{equation}\label{Minclud}
		M\cap RB^n \supset \sum_{i=0}^N \beta_i D_i(R).
		\end{equation}
	Since $D_i(R)\subset h_i a+E$ for $i=0,\dots,N$, we have $D_i(R)-h_i a\subset E$, and thus, by the translation invariance of $\mathcal H^m$ and the Brunn--Minkowski inequality in $E$,
\begin{equation}\label{eq:BM-DiR}
	\begin{aligned}
		\mathcal H^m\!\left(\sum_{i=0}^N \beta_i D_i(R)\right)^{1/m}
		&=
		\mathcal H^m\!\left(\sum_{i=0}^N \beta_i\bigl(D_i(R)-h_i a\bigr)\right)^{1/m} \\
		&\ge
		\sum_{i=0}^N \beta_i\,\mathcal H^m\!\left(D_i(R)-h_i a\right)^{1/m} \\
		&=
		\sum_{i=0}^N \beta_i\,\mathcal H^m\!\left(D_i(R)\right)^{1/m}.
	\end{aligned}
\end{equation}
		On the other hand, each $h_i$ is orthogonal  so
		\[
		D_i(R)=h_i\bigl(C\cap RB^n\bigr),
		\]
		and hence
	\begin{equation}\label{DC}
		\mathcal H^m\bigl(D_i(R)\bigr)=\mathcal H^m\bigl(C\cap RB^n\bigr)
		\qquad\text{for all }i=0,\dots,N.
		\end{equation}
		Therefore, by \eqref{Minclud}, \eqref{eq:BM-DiR}, and \eqref{DC},
		\[
		\mathcal H^m(M\cap RB^n)\ge \mathcal H^m\!\left(\sum_{i=0}^N \beta_i D_i(R)\right)\ge \mathcal H^m\bigl(C\cap RB^n\bigr)
		\qquad\text{for every }R>0.
		\]
	By the layer-cake representation, we obtain
		\begin{equation}\label{eq:fubini-comparison}
			\begin{aligned}
				\int_M |x|^{\,q-n}\,d\mathcal H^m(x)
				&=
				\int_0^\infty \mathcal H^m\bigl(\{x\in M:\ |x|^{\,q-n}>R\}\bigr)\,dR \\
				&=\int_0^\infty \mathcal H^m(M\cap R^{\frac{1}{q-n}}B^n)\,dR \\
				&=
				(n-q)\int_0^\infty R^{\,q-n-1}\,\mathcal H^m(M\cap RB^n)\,dR \\
				&\ge
				(n-q)\int_0^\infty R^{\,q-n-1}\,\mathcal H^m(C\cap RB^n)\,dR \\
				&=
				\int_C |x|^{\,q-n}\,d\mathcal H^m(x).
			\end{aligned}
		\end{equation}
	If equality holds in \eqref{eq:fubini-comparison},	we conclude that
		\[
		\mathcal H^m(M\cap RB^n)=\mathcal H^m(C\cap RB^n)
		\qquad \text{for every } R>0.
		\]
			This completes the proof.
		\end{proof}
		
		\begin{remark}\label{remq=n}
				If $q=n$, we choose $R>0$ so large that both $C$ and $M$ are contained in $RB^n$. Then the same ball-section comparison gives
			\[
			\mathcal H^m(M)=\mathcal H^m(M\cap RB^n)\ge \mathcal H^m(C\cap RB^n)=\mathcal H^m(C).
			\]
		\end{remark}
		
In the origin-symmetric setting, dual curvature measures satisfy the $q$-th subspace mass inequality for $0<q<n$; see \cite{BHP2018}. We now consider the corresponding behavior in the $G$-invariant framework.
	
	\begin{theorem}\label{thm:G-invariant-subspace-mass}
Let $G\subset O(n)$ be a closed subgroup with no nonzero fixed points.	Let $K\in \KG$  and let $q\in(0,n]$. If $L\subset \mathbb R^n$ is a proper $G$-invariant subspace, then
		\begin{equation}\label{subspace mass}
		\frac{\widetilde C_q(K,S^{n-1}\cap L)}{\widetilde C_q(K,S^{n-1})}
		\le
		\min\left\{\frac{\dim L}{q},1\right\}.
		\end{equation}
		Moreover, if $q<n$, then the inequality is strict.
	\end{theorem}
	
	\begin{proof}
		Write $k=\dim L$ and $m=n-k$. If $k\ge q$, then the desired estimate is immediate, because the right-hand side is equal to $1$. Hence it remains to treat the nontrivial case $0<k<q$.
		
		For $u\in \SL$, define
		\begin{equation}\label{symbol}
		\rho(u):=\rho_{K|L}(u),\qquad y_u:=\rho(u)u\in \partial(K|L),\qquad A_u:=K\cap (y_u+\lp).
		\end{equation}
		
		We first estimate the total mass. By the integral representation \eqref{Cq representation} of $\widetilde C_q(K,\cdot)$, we have
		\[
		\widetilde C_q(K,S^{n-1})=\frac{q}{n}\int_K |x|^{\,q-n}\,d\mathcal H^n(x).
		\]
		Fubini's theorem and polar coordinates in $L$ yield
		\begin{equation}\label{eq:total-mass-decomposition}
			\begin{aligned}
				\widetilde C_q(K,S^{n-1})
				&=
				\frac{q}{n}\int_{K|L}
				\left(
				\int_{K\cap(y+L^\perp)} |z|^{\,q-n}\,d\mathcal H^m(z)
				\right)
				\,d\mathcal H^k(y) \\
				&=
				\frac{q}{n}\int_{\SL}\int_0^{\rho(u)}
				\left(
				\int_{K\cap(ru+\lp)} |z|^{\,q-n}\,d\mathcal H^m(z)
				\right)
				r^{k-1}\,dr\,d\mathcal H^{k-1}(u).
			\end{aligned}
		\end{equation}
		
		Fix $u\in \SL$. Since $L$ is $G$-invariant and $y_u\in \partial(K|L)$, we have $Gy_u\subset L$. The centroid of the compact convex set $\operatorname{conv}(Gy_u)$ is $G$-invariant, and hence it is equal to $0$ because $\mathrm{Fix}(G)=\{0\}$. Therefore,
		\[
		0\in \operatorname{conv}(Gy_u).
		\] 
		Then there exist  $g_1,\dots,g_N\in G$ and  $\lambda_1,\dots,\lambda_N\ge 0$ such that
			\begin{equation}\label{centroid}
		\sum_{i=1}^N \lambda_i=1,
		\qquad
		\sum_{i=1}^N \lambda_i g_i y_u=0.
			\end{equation}
	Recall that $A_u=K\cap (y_u+\lp)$.	Define
		\[
		B_u:=\sum_{i=1}^N \lambda_i\,g_i A_u.
		\]
		Since $\lp$ is $G$-invariant by Lemma \ref{lem:orthogonal-complement-invariant},  each $g_iA_u$ is contained in the affine subspace $g_i y_u+\lp$, and therefore any point of $B_u$ has $L$-component
		\[
		\sum_{i=1}^N \lambda_i g_i y_u=0.
		\]
		Since $K$ is convex and $G$-invariant, we conclude that
		\begin{equation}\label{Bu}
		B_u\subset K\cap \lp.
		\end{equation}
	
		Now let $r\in[0,\rho(u)]$ and put $t=r/\rho(u)$. By the convexity of $K$,
		\[
	tA_u+(1-t)B_u\subset K.
		\]
		Moreover, by \eqref{Bu}, every point of $tA_u+(1-t)B_u$ has $L$-component 
		\[
		ty_u+(1-t)\cdot 0=ru,
		\]
		so
		\begin{equation}\label{AK}
		tA_u+(1-t)B_u\subset K\cap(ru+\lp).
		\end{equation}
		
		We now apply Lemma~\ref{lem:parallel-sections}, together with Remark~\ref{remq=n}. Indeed, $A_u\subset y_u+\lp$, $g_iA_u\subset g_iy_u+\lp$.  Taking the coefficients
		\[
		\beta_0=t,\qquad \beta_i=(1-t)\lambda_i,\quad i=1,\dots,N,
		\]
		whose sum is equal to $1$, the lemma yields
		\begin{equation}\label{equatity1}
		\int_{tA_u+(1-t)B_u} |z|^{\,q-n}\,d\mathcal H^m(z)\ge \int_{A_u}|z|^{\,q-n}\,d\mathcal H^m(z)
		\end{equation}
	for every $u\in \SL$ and every $r\in[0,\rho(u)]$. Hence, by \eqref{AK},
			\begin{equation}\label{equatity2}
		\int_{K\cap(ru+\lp)} |z|^{\,q-n}\,d\mathcal H^m(z)\ge \int_{A_u}|z|^{\,q-n}\,d\mathcal H^m(z)
		\end{equation}
		for every $u\in \SL$ and every $r\in[0,\rho(u)]$.
	Substituting this estimate into \eqref{eq:total-mass-decomposition}, we obtain
	\begin{equation}\label{greatthan}
		\begin{aligned}
			\widetilde C_q(K,S^{n-1})
			&\ge
			\frac{q}{n}\int_{\SL}\int_0^{\rho(u)} \left( \int_{A_u}|z|^{\,q-n}\,d\mathcal H^m(z)\right) \,r^{k-1}\,dr\,d\mathcal H^{k-1}(u) \\
			&=
			\frac{q}{nk}\int_{\SL} \left( \int_{A_u}|z|^{\,q-n}\,d\mathcal H^m(z)\right) \rho(u)^k\,d\mathcal H^{k-1}(u).
		\end{aligned}
	\end{equation}
	
		We next compute the restriction of $\widetilde C_q(K,\cdot)$ to $\SL$. For $x\in K\setminus\{0\}$, we have $x/|x|\in \alpha_K^*(\SL)$ if and only if $\rho_{K}(x)x$ has an outer normal in $L$. By Lemma \ref{lem:fiber-identity},
		\begin{equation}\label{eq:alpha-star-decomposition}
			\begin{aligned}
				\{x\in K\setminus\{0\}: x/|x|\in \alpha_K^*(\SL)\}\cup\{0\}
				&=
				\bigcup_{u\in\SL}\operatorname{conv}\{0,A_u\} \\
				&=
				\bigcup_{u\in\SL}\bigcup_{0\le r\le \rho(u)} \frac{r}{\rho(u)}A_u .
			\end{aligned}
		\end{equation}
		Using the  fiber identity \eqref{eq:fiber-section},  \eqref{Cq representation},  Fubini's theorem, and polar coordinates in the $k$-dimensional space $L$, we obtain
		\begin{align*}
			\widetilde C_q(K,\SL)
			&=
			\frac{q}{n}\int_{	\{x\in K\setminus\{0\}: x/|x|\in \alpha_K^*(\SL)\}}|x|^{\,q-n}\,d\mathcal H^n(x) \\
			&=	\frac{q}{n}\int_{K|L}\int_{	\{x\in K\setminus\{0\}: x/|x|\in \alpha_K^*(\SL)\}\cap( y+\lp) }|z|^{\,q-n}\,d\mathcal H^m(z) d\mathcal H^k(y) \\
			&=
			\frac{q}{n}\int_{\SL}\int_0^{\rho(u)}
			\left(
			\int_{(r/\rho(u))A_u} |z|^{\,q-n}\,d\mathcal H^m(z)
			\right)
			r^{k-1}\,dr\,d\mathcal H^{k-1}(u)\\
			&=
			\frac{q}{n}\int_{\SL}\int_0^{\rho(u)}
			\left(\frac{r}{\rho(u)}\right)^{q-k} 	\left( \int_{A_u} |z|^{\,q-n}\,d\mathcal H^m(z)\right) \,r^{k-1}\,dr\,d\mathcal H^{k-1}(u) \\
			&=\frac{1}{n}\int_{\SL} \left( \int_{A_u} |z|^{\,q-n}\,d\mathcal H^m(z)\right)\rho(u)^k\,d\mathcal H^{k-1}(u).
		\end{align*}

		Comparing this with the formula \eqref{greatthan}, we conclude that
		\[
		\widetilde C_q(K,\SL)\le \frac{k}{q}\,\widetilde C_q(K,S^{n-1}),
		\]
		which proves the asserted inequality in the nontrivial range $k<q$.

		Assume that equality holds in \eqref{subspace mass} for some proper $G$-invariant subspace $L$ and $q<n$. Since $\widetilde C_q(K,\cdot)$ is not concentrated on any proper subsphere, we must have $\dim L<q$. Hence equality must hold in \eqref{greatthan}. Therefore, equality holds in \eqref{equatity2} for $\mathcal H^{k-1}$-almost every $u\in \SL$ and for almost every $r\in(0,\rho(u))$.
		
		For such $u$ and $r$, by \eqref{AK}, \eqref{equatity1}, \eqref{equatity2}, and the equality case in Lemma~\ref{lem:parallel-sections}, we obtain
		\[
		\mathcal H^m\bigl(A_u\cap RB^n\bigr)
		=
		\mathcal H^m\bigl((K\cap(ru+L^\perp))\cap RB^n\bigr)
		\qquad\text{for every }R>0.
		\]
		Since $0\in \operatorname{int}K$, there exists $\delta>0$ such that
		$\delta B^n\subset K$. Now fix such a pair $(u,r)$ satisfying
		\[
		0<r<\min\{\rho(u),\delta\}.
		\] Since $A_u\subset y_u+L^\perp$ and $|y_u|=\rho(u)$, we have
		\[
		A_u\cap RB^n=\varnothing
		\qquad\text{whenever }0<R<\rho(u).
		\]
	 Choosing $R$ such that
		\[
		0<r<R<\min\{\rho(u),\delta\},
		\]
		we see that
		\[
		(ru+L^\perp)\cap RB^n
		\]
		has positive $m$-dimensional volume and is contained in $K\cap(ru+L^\perp)$. Hence
		\[
		\mathcal H^m\bigl((K\cap(ru+L^\perp))\cap RB^n\bigr)>0,
		\]
		which is impossible. This contradiction proves that the inequality is strict when $q<n$. 
		

	\end{proof}
	
	We call the mass distribution property appearing in Theorem \ref{thm:G-invariant-subspace-mass} the $G$-invariant $q$-th subspace mass inequality. More precisely, we make the following definition.
		\begin{definition}
		Let $q>0$, let $G\subset O(n)$, and let $\mu$ be a nonzero finite Borel measure on $\Sph$. We say that $\mu$ satisfies the \textit{$G$-invariant $q$-th subspace mass inequality} if
		\begin{equation}\label{G qth subspace mass}
			\frac{\mu(\Sph\cap L)}{\mu(\Sph)}
			<
			\begin{cases}
				\dfrac{\dim L}{q}, & \dim L<q,\\[1ex]
				1, & \dim L\ge q,
			\end{cases}
		\end{equation}
		for every proper $G$-invariant subspace $L\subset \R^n$.
	\end{definition}
	The $G$-invariant $q$-th subspace mass inequality only tests the mass of the measure on $G$-invariant subspaces. Note that there may be very few such subspaces. For example, if $G$ is irreducible, then there are no nontrivial $G$-invariant subspaces. Therefore, \eqref{G qth subspace mass} carries no information in this case. Furthermore, let
	\[
	\R^n=V_1\oplus\cdots\oplus V_m
	\]
	be an orthogonal decomposition, and let
	\begin{equation}\label{product group}
		G=G_1\times\cdots\times G_m\subset O(V_1)\times\cdots\times O(V_m)\subset O(n),
	\end{equation}
	where each $G_j$ is irreducible on $V_j$ (when $\dim V_j=1$ we assume $G_j\neq \{I_{V_j}\}$). Then the $G$-invariant subspaces of $\mathbb R^n$ are precisely
	$
	\bigoplus_{j\in I}V_j
	$
	for some subset $I\subset\{1,\dots,m\}$. At the other extreme, when $G=\{\pm I\}$, every subspace of $\mathbb R^n$ is $\{\pm I\}$-invariant. Hence the $\{\pm I\}$-invariant $q$-th subspace mass inequality is exactly the classical $q$-th subspace mass inequality.

	Next, we consider a special class of symmetries, namely irreducible ones. This type of symmetry is highly rigid (see \cite{shan2026}).
	
	\begin{proposition}\label{irreducible-subspace-mass}
		Let $G\subset O(n)$ be an irreducible subgroup, and let $\mu$ be a nonzero finite
		$G$-invariant Borel measure on $\Sph$. Then, for every proper subspace
		$\xi\subset \R^n$,
		\[
		\frac{\mu( \Sph \cap \xi)}{\mu(\Sph)}
		\le
		\frac{\dim \xi}{n}.
		\]
		Moreover, equality holds if and only if
		\[
		\mu\bigl(\Sph\setminus(\xi\cup \xi^\perp)\bigr)=0.
		\]
	\end{proposition}
	
	\begin{proof}
		Define
		\begin{equation}\label{defm}
			M:=\int_{\Sph} x\otimes x\,d\mu(x),
		\end{equation}
		where $(x\otimes x)v=\langle x,v\rangle x$ for $v\in\R^n$. For every
		$g\in G$ and every $v\in\R^n$,
		\[
		(g(x\otimes x)g^{-1})v
		=
		g\bigl((x\otimes x)(g^{-1}v)\bigr)
		=
		g\bigl(\langle x,g^{-1}v\rangle x\bigr)
		=
		\langle gx,v\rangle gx
		=
		\bigl((gx)\otimes(gx)\bigr)v.
		\]
		Hence, by the $G$-invariance of $\mu$,
		\[
		gMg^{-1}
		=
		\int_{\Sph} (gx)\otimes(gx)\,d\mu(x)
		=
		\int_{\Sph} y\otimes y\,d\mu(y)
		=
		M.
		\]
		In other words,
		\begin{equation}\label{commute}
			gM=Mg
			\qquad\text{for every }g\in G.
		\end{equation}
		
		Since $M$ is symmetric, all its eigenvalues are real. Let $\lambda$ be an
		eigenvalue of $M$, and let
		\[
		V_\lambda=\{v\in\R^n:Mv=\lambda v\}
		\]
		be the corresponding eigenspace. For any $v\in V_\lambda$ and any $g\in G$,
		by \eqref{commute} we have
		\[
		M(gv)=g(Mv)=g(\lambda v)=\lambda(gv).
		\]
		Thus $gv\in V_\lambda$, and so $V_\lambda$ is $G$-invariant. By the
		irreducibility of $G$, we have $V_\lambda=\R^n$. Therefore $M=\lambda I_n$.
		Taking traces gives
		\[
		n\lambda=\operatorname{tr}(M)
		=
		\sum_{i=1}^n\int_{\Sph}x_i^2\,d\mu(x)
		=
		\int_{\Sph}|x|^2\,d\mu(x)
		=
		\mu(\Sph),
		\]
		since $|x|=1$ on $\Sph$. Therefore
		\begin{equation}\label{identityM}
			M=\frac{\mu(\Sph)}{n}\,I_n.
		\end{equation}
		
	Now let $P_\xi$ be the orthogonal projection onto $\xi$. Let $k=\dim\xi$.
	Choose an orthonormal basis $e_1,\dots,e_k$ of $\xi$, and extend it to an
	orthonormal basis $e_1,\dots,e_n$ of $\R^n$. With respect to this basis, the matrix
	of $P_\xi$ is
	\[
	P_\xi=
	\begin{pmatrix}
		I_k & 0\\
		0 & 0
	\end{pmatrix}.
	\]
	For $x\in\Sph$, write
	\[
	x=\sum_{i=1}^n x_i e_i .
	\]
	Then $P_\xi x=\sum_{i=1}^k x_i e_i$, and hence
	\begin{equation}\label{eq:projection-trace}
		\int_{\Sph}|P_\xi x|^2\,d\mu(x)
		=
		\int_{\Sph}\sum_{i=1}^k x_i^2\,d\mu(x)
		=
		\operatorname{tr}(P_\xi M).
	\end{equation}
	Using \eqref{identityM}, we obtain
	\begin{equation}\label{projection}
		\int_{\Sph}|P_\xi x|^2\,d\mu(x)
		=
		\frac{\mu(\Sph)}{n}\operatorname{tr}(P_\xi)
		=
		\frac{\dim\xi}{n}\,\mu(\Sph),
	\end{equation}
	since $\operatorname{tr}(P_\xi)=\dim\xi$.
		
		On the other hand, if $x\in\xi\cap\Sph$, then $|P_\xi x|=1$. Therefore,
		\begin{equation}\label{proine}
			\mu(\xi\cap\Sph)
			=
			\int_{\xi\cap\Sph}1\,d\mu(x)
			=
			\int_{\xi\cap\Sph}|P_\xi x|^2\,d\mu(x)
			\le
			\int_{\Sph}|P_\xi x|^2\,d\mu(x).
		\end{equation}
		Combining this with \eqref{projection} yields
		\[
		\frac{\mu(\xi\cap\Sph)}{\mu(\Sph)}
		\le
		\frac{\dim\xi}{n}.
		\]
		
		It remains to characterize the equality case. By \eqref{proine}, equality holds
		if and only if
		\[
		\int_{\Sph\setminus\xi}|P_\xi x|^2\,d\mu(x)=0.
		\]
		Since the integrand is nonnegative, this is equivalent to
		\[
		|P_\xi x|^2=0
		\qquad\text{for $\mu$-a.e. }x\in\Sph\setminus\xi.
		\]
		Since $|P_\xi x|=0$ if and only if $x\in\xi^\perp$, equality holds if and only if
		\[
		\mu\bigl(\Sph\setminus(\xi\cup\xi^\perp)\bigr)=0.
		\]
		This completes the proof.
	\end{proof}
	
Proposition~\ref{irreducible-subspace-mass} shows that irreducible symmetry is strong enough to force the classical subspace concentration condition. More precisely, every nonzero finite Borel measure on $\Sph$ invariant under an irreducible subgroup of $O(n)$ automatically satisfies this condition. Moreover,  if equality holds for some proper subspace $\xi$, then the measure is concentrated on $\Sph \cap (\xi\cup\xi^\perp)$.

\begin{remark}\label{rem:irreducible-vs-fixed}
		For $n\ge 2$, every irreducible subgroup of $O(n)$ has no nonzero fixed points.
The condition of having no nonzero fixed points is substantially weaker than
irreducibility. This gap becomes more visible in higher dimensions. 

First, for any $G\subset O(n)$, the space $\R^n$ can be
	decomposed orthogonally into $G$-invariant irreducible blocks,
	\[
	\R^n=U_1\oplus\cdots\oplus U_m .
	\]
The group $G$ is irreducible  precisely when $m=1$, whereas the condition $\Fix(G)=\{0\}$ only means that no  one-dimensional block
	$U_i$ is fixed pointwise by all elements of $G$. 
	
	This difference can also be seen from the orbit criterion. By \cite{shan2026},
	$G$ is irreducible if and only if, for every $v\in S^{n-1}$, the orbit $Gv$ is not
	contained in any closed hemisphere of $S^{n-1}$. In contrast, $G$ has no nonzero
	fixed points if and only if $Gv$ is not a singleton for every $v\in S^{n-1}$.
	
	The same distinction appears in the structure of John ellipsoids. By  \cite{shan2026}, if $G$ is irreducible and $K\in\KG$, then the John ellipsoid of $K$ is a Euclidean ball centered at the origin. Under the weaker assumption that $G$ has no nonzero fixed points, the John ellipsoid is only an ellipsoid centered at the origin; see Section~4.
	
	When $G=\{\pm I\}$, one has $\Fix(G)=\{0\}$, and every subspace of $\R^n$ is
	$G$-invariant. We next record several further examples.
	
Product groups provide many such examples. Let
\[
\R^n=V_1\oplus\cdots\oplus V_m,\qquad m\ge 2,
\]
be an orthogonal decomposition into nonzero subspaces, and suppose that $G_i\subset O(V_i)$ satisfies $\Fix(G_i)=\{0\}$ for each $i$. Then the
 product
\[
G=G_1\times\cdots\times G_m\subset O(n)
\]
satisfies
\[
\Fix(G)=\Fix(G_1)\oplus\cdots\oplus\Fix(G_m)=\{0\},
\]
but $G$ is reducible since each $V_i$ is $G$-invariant.
	
When the group acts diagonally on repeated blocks rather than independently on
each block, additional invariant subspaces may appear. For example, write
	\[
	\R^4=\R^2\oplus\R^2,
	\]
	and let $G\subset SO(4)$ be the diagonal rotation group
	\[
	G=\left\{
	\begin{pmatrix}
		R_\theta & 0\\
		0 & R_\theta
	\end{pmatrix}
	:\theta\in\R
	\right\},
	\]
	where $R_\theta$ denotes the rotation of $\R^2$ by angle $\theta$. Then
	$\Fix(G)=\{0\}$, but the action is reducible, since
	$\R^2\oplus\{0\}$ and $\{0\}\oplus\R^2$ are $G$-invariant. Moreover, there are
	continuously many two-dimensional $G$-invariant subspaces. Indeed, for
	$a,b\in\R$, put
	\[
	A_{a,b}=
	\begin{pmatrix}
		a & -b\\
		b & a
	\end{pmatrix},
	\qquad
	L_{a,b}=\{(u,A_{a,b}u):u\in\R^2\}.
	\]
	Since $A_{a,b}R_\theta=R_\theta A_{a,b}$ for every $\theta$, each $L_{a,b}$ is
	$G$-invariant. 
	
	There are also simple examples in $SO(3)$. Let
	\[
	G=\{I,\operatorname{diag}(1,-1,-1),
	\operatorname{diag}(-1,1,-1),
	\operatorname{diag}(-1,-1,1)\}\subset SO(3).
	\]
	This is the Klein four group generated by rotations of angle $\pi$ about the
	coordinate axes. It has no nonzero fixed points, since a vector
	$x=(x_1,x_2,x_3)$ fixed by all elements of $G$ must satisfy
	$x_1=x_2=x_3=0$. Nevertheless, the action is highly reducible: each coordinate
	line $\R e_i$ is $G$-invariant.  
	
More generally,	 for any
	$n\ge 2$, let
	\[
	G=\left\{
	\operatorname{diag}(\varepsilon_1,\ldots,\varepsilon_n):
	\varepsilon_i=\pm 1,\ 
	\varepsilon_1\cdots\varepsilon_n=1
	\right\}
	\subset SO(n).
	\]
	This group consists of all diagonal sign changes with an even number of minus
	signs. It has no nonzero fixed points.  Nevertheless, the action is highly reducible, since every coordinate line, and
	more generally every coordinate subspace, is $G$-invariant.
	
A more geometric example is the rotational symmetry group of a regular
$m$-gonal prism. Let $m\ge 3$, let $R$ be the rotation about the $e_3$-axis by
angle $2\pi/m$, and let $S$ be the rotation about the $e_1$-axis by angle $\pi$.
The group $G=\langle R,S\rangle\subset SO(3)$ preserves the decomposition
\[
\R^3=\operatorname{span}\{e_1,e_2\}\oplus \R e_3,
\]
and hence is reducible. On the other hand, $\Fix(R)=\R e_3$, while $S$ sends
$e_3$ to $-e_3$. Hence $\Fix(G)=\{0\}$.
	
	If $G\subset O(n)$ is connected and $L=\R v$ is a one-dimensional $G$-invariant subspace
	with $|v|=1$, then the orbit $Gv$ is connected. Since $L$ is $G$-invariant and $G\subset O(n)$, we have $Gv\subset\{v,-v\}$. Hence $Gv=\{v\}$, and
	therefore $L\subset\Fix(G)$. Thus, for connected groups,
	the condition $\Fix(G)=\{0\}$ rules out one-dimensional $G$-invariant subspaces. Consequently, for connected subgroups $G\subset O(2)$, the condition
	$\Fix(G)=\{0\}$ is equivalent to the irreducibility of $G$.

	These examples show that the passage from irreducible groups to groups with
	no nonzero fixed points is substantial. 
\end{remark}

\begin{remark}
It was shown in \cite{shan2026} that $\mathcal K_G \subset \mathcal K_e^n$
if and only if $-x\in \overline{Gx}$ for all $x\in S^{n-1}$. In particular,
if $G$ is finite, then $\mathcal K_G \subset \mathcal K_e^n$ if and only if
$-I\in G$. Hence there are many groups for which $G$-invariant convex bodies
need not be origin-symmetric.
\end{remark}


\section{Structure of $G$-invariant ellipsoids}

	Ellipsoids play an important role in variational approaches to the existence theory of Minkowski-type problems. In this section, we investigate several properties of $G$-invariant ellipsoids.  Throughout this section, let $G\subset O(n)$ be a closed subgroup with no nonzero fixed points.

	We first record the basic facts about  John ellipsoids. Recall that, for a convex body $K\subset \mathbb R^n$, the John ellipsoid of $K$ is the unique ellipsoid of maximal volume contained in $K$.
	
	\begin{lemma}\label{lem:John}
		Let $K\in  \KG$. Then the John ellipsoid $Q$ of $K$ is $G$-invariant and centered at the origin. In particular, there exists a positive definite  operator $A$ commuting with every element of $G$ such that
		\[
		Q=\{x\in \R^n:\ip{Ax}{x}\le 1\}.
		\]
		If $\lambda_1>\cdots>\lambda_m>0$ are the distinct eigenvalues of $A$, and $V_1,\dots,V_m$ are the corresponding eigenspaces, then
		\[
		\R^n=V_1\oplus\cdots\oplus V_m
		\]
		is an orthogonal decomposition into $G$-invariant subspaces, and, setting $b_j=\lambda_j^{-1/2}$, one has
		\[
		Q=\left\{x\in \R^n:\sum_{j=1}^m \frac{\abs{P_{V_j}x}^2}{b_j^2}\le 1\right\}.
		\]
	\end{lemma}
	
	\begin{proof}
		The John ellipsoid of $K$ is unique. If $g\in G$, then $gQ\subset gK=K$, and $\vol(gQ)=\vol(Q)$ because $g$ is orthogonal. Hence $gQ$ is also the John ellipsoid of $K$, so $gQ=Q$. Thus $Q$ is $G$-invariant.
		
Since $Q$ is $G$-invariant, its center is fixed by every element of $G$. Since $\mathrm{Fix}(G)=\{0\}$, the center of $Q$ must be the origin. Therefore,
\begin{equation}\label{ellipsoid}
Q=\{x\in \mathbb R^n:\langle Ax,x\rangle\le 1\}
\end{equation}
for some symmetric positive definite matrix $A\in \mathrm{GL}(n, \R)$.  For any $g\in G$, since $g^{-1}=g^\top$, we have
$$
gQ
=
\{gx:\langle Ax,x\rangle\le 1\}
=
\{y:\langle A(g^{-1}y),g^{-1}y\rangle\le 1\}
=
\{y:\langle gAg^{-1}y,y\rangle\le 1\}.
$$
	Equating $gQ = Q$ gives $gAg^{-1} = A$. 
	In other words, 
	\begin{equation}\label{commute2}
		Ag = gA  \quad \text{for all} \quad g \in G.
	\end{equation}

	Since $A$ is symmetric positive definite, all its eigenvalues are positive. Let
	\[
	\lambda_1>\cdots>\lambda_m>0
	\]
	be the distinct eigenvalues of $A$, and let
	\begin{equation}\label{eigenspace}
	V_j=\{v\in \mathbb R^n:Av=\lambda_j v\},
	\qquad j=1,\dots,m,
	\end{equation}
	be the corresponding eigenspaces. By the spectral theorem,
	\begin{equation}\label{ellipsoid od}
	\mathbb R^n=V_1\oplus\cdots\oplus V_m
	\end{equation}
	is an orthogonal decomposition.
	 Let $v\in V_j$ and let $g\in G$. By \eqref{commute2},
	\[
	A(gv)=g(Av)=g(\lambda_j v)=\lambda_j(gv).
	\]
	Thus $gv\in V_j$, and so $V_j$ is $G$-invariant.
	
	Finally, for any $x\in \mathbb R^n$, write
	\[
	x=P_{V_1}x+\cdots+P_{V_m}x.
	\]
By \eqref{eigenspace} and \eqref{ellipsoid od} we obtain
	\[
	\langle Ax,x\rangle
	=
	\sum_{j=1}^m \lambda_j |P_{V_j}x|^2.
	\]
	Setting
	\[
	b_j=\lambda_j^{-1/2},
	\qquad j=1,\dots,m,
	\]
	it follows from \eqref{ellipsoid} that
	\[
	Q
	=
	\left\{x\in \mathbb R^n:\sum_{j=1}^m \lambda_j |P_{V_j}x|^2\le 1\right\}
	=
	\left\{x\in \mathbb R^n:\sum_{j=1}^m \frac{|P_{V_j}x|^2}{b_j^2}\le 1\right\}.
	\]
	\end{proof}
	
	\begin{remark}
If $G$ has a nonzero fixed point $v$, then a $G$-invariant convex body, and hence its John ellipsoid, need not be origin-symmetric and may even fail to contain the origin. For example, for any $t>0$, the translated ball
		\[
		B^n+t v
		\]
		is $G$-invariant. 
	\end{remark}

The following lemma records some basic facts about the convergence of $G$-invariant subspace blocks.
	
\begin{lemma}\label{lem:Grass}
	Fix integers $d_1,\dots,d_m\ge 1$ such that $d_1+\cdots+d_m=n$. For each $\ell\in\mathbb N$, let
	\[
	\R^n=V_{1,\ell}\oplus\cdots\oplus V_{m,\ell}
	\]
	be an orthogonal decomposition into $G$-invariant subspaces with $\dim V_{j,\ell}=d_j$ for $j=1,\dots,m$. Assume that, for each $j$, the subspaces $V_{j,\ell}$ converge in $\Gr(n,d_j)$ to a subspace $V_j$ as $\ell\to\infty$. Then each $V_j$ is $G$-invariant, the subspaces $V_1,\dots,V_m$ are pairwise orthogonal, and
	\[
	\R^n=V_1\oplus\cdots\oplus V_m.
	\]
\end{lemma}
	
	\begin{proof}
	Let
		\[
		P_{j,\ell}:=P_{V_{j,\ell}},\qquad P_j:=P_{V_j}, \qquad j=1,\dots,m.
		\]
		By definition of Grassmannian convergence \eqref{grassmaiann metric},
		\[
		\|P_{j,\ell}-P_j\|\to 0
		\qquad
		(\ell\to\infty).
		\]
		
	Since $V_{j,\ell}$ is $G$-invariant, Lemma~\ref{lem:projection-invariant} gives
	\[
	gP_{j,\ell}=P_{j,\ell}g
	\qquad\text{for every }g\in G.
	\]
	Passing to the limit in operator norm yields
	\[
	gP_j=P_jg
	\qquad\text{for every }g\in G.
	\]
	By Lemma~\ref{lem:projection-invariant} again, $V_j$ is $G$-invariant.
		
		Because the decomposition for each $\ell$ is orthogonal,
		\[
		P_{j,\ell}P_{k,\ell}=0\qquad\text{whenever }j\ne k.
		\]
		Passing to the limit gives
		\[
		P_jP_k=0\qquad\text{whenever }j\ne k.
		\]
		Therefore $V_j\perp V_k$ for $j\ne k$.
		
		Finally,
		\[
		P_{1,\ell}+\cdots+P_{m,\ell}=I
		\]
		for every $\ell$, and passing to the limit yields
		\[
		P_1+\cdots+P_m=I.
		\]
		Thus every $x\in \R^n$ can be written as
		\[
		x=P_1x+\cdots+P_mx,
		\]
		with $P_jx\in V_j$. Hence $\R^n=V_1+\cdots+V_m$. Since the subspaces are pairwise orthogonal, this sum is direct.
	\end{proof}

	\vspace{1.7em}
	We need the following spherical partition associated with the subspace blocks. Let
	\[
	\R^n=V_1\oplus\cdots\oplus V_m
	\]
	be an orthogonal decomposition into nonzero subspaces. For $\delta\in(0,1/\sqrt m)$ and $j=1,\dots,m$, define
	\begin{equation}\label{partation}
	\Omega_{j,\delta}
	=
	\left\{
	u\in \Sph:
	|P_{V_j}u|\ge \delta
	\text{ and }
	|P_{V_t}u|<\delta \text{ for all }t>j
	\right\}.
	\end{equation}
	Then the family $\{\Omega_{1,\delta},\dots,\Omega_{m,\delta}\}$ forms a partition of $\Sph$. Indeed, for every $u\in \Sph$,
	\[
	1=|u|^2=\sum_{j=1}^m |P_{V_j}u|^2.
	\]
	If $|P_{V_j}u|<\delta$ for all $j$, then
	\[
	1=\sum_{j=1}^m |P_{V_j}u|^2<m\delta^2<1,
	\]
	a contradiction. Hence there exists at least one index $j$ such that $|P_{V_j}u|\ge \delta$. Let $j_0$ be the largest such index. Then
	\[
	|P_{V_{j_0}}u|\ge \delta
	\qquad\text{and}\qquad
	|P_{V_t}u|<\delta \text{ for all }t>j_0,
	\]
	so $u\in \Omega_{j_0,\delta}$. Thus
	\begin{equation}\label{eq:Omega-partition}
		\Sph=\Omega_{1,\delta}\cup\cdots\cup\Omega_{m,\delta}.
	\end{equation}
	Moreover, the sets $\Omega_{1,\delta},\dots,\Omega_{m,\delta}$ are clearly pairwise disjoint by definition.
	
	Now set
	\[
	W_r:=V_1\oplus\cdots\oplus V_r,
	\qquad r=1,\dots,m.
	\]
	Then for every $r\in\{1,\dots,m\}$,
	\begin{equation}\label{eq:Omega-union}
		\Omega_{1,\delta}\cup\cdots\cup\Omega_{r,\delta}
		=
		\left\{
		u\in \Sph:
		|P_{V_t}u|<\delta \text{ for all }t>r
		\right\}.
	\end{equation}
	In particular,
	\begin{equation}\label{eq:Omega-decrease}
		\Omega_{1,\delta}\cup\cdots\cup\Omega_{r,\delta}
		\downarrow
		W_r\cap \Sph
		\qquad\text{as }\delta\to 0^{+}.
	\end{equation}
	Indeed, if $u\in W_r\cap \Sph$, then $P_{V_t}u=0$ for all $t>r$, so by \eqref{eq:Omega-union},
	\[
	u\in \Omega_{1,\delta}\cup\cdots\cup\Omega_{r,\delta}
	\qquad\text{for every }\delta>0.
	\]
	Conversely, if
	\[
	u\in \bigcap_{\delta>0}\bigl(\Omega_{1,\delta}\cup\cdots\cup\Omega_{r,\delta}\bigr),
	\]
	then by \eqref{eq:Omega-union},
	$
	|P_{V_t}u|<\delta$
	for all $t>r$
	and for all $\delta>0$, which implies $P_{V_t}u=0$ for all $t>r$. Hence $u\in W_r\cap \Sph$.
	
	Consequently, for every finite Borel measure $\mu$ on $\Sph$,
	\begin{equation}\label{eq:Omega-measure-limit}
		\lim_{\delta\to 0^{+}}
		\mu\bigl(\Omega_{1,\delta}\cup\cdots\cup\Omega_{r,\delta}\bigr)
		=
		\mu(W_r\cap \Sph),
	\end{equation}
	by the continuity from above of finite measures.
	
		\vspace{1.4em}

	For a nonzero finite Borel measure $\mu$ on $\Sph$, define the entropy functional
	\[
	E_\mu(K)=-\frac1{\abs{\mu}}\int_{\Sph}\log h_K(v)\,d\mu(v),\qquad K\in \mathcal K_o^n.
	\]
To solve the dual Minkowski problem by variational methods, one needs to estimate the limiting behavior of the entropy. Entropy estimates in the origin-symmetric setting involving ellipsoids can be found, for example, in \cite{BLYZZ2019,HLYZdual}. We now prove the entropy estimate adapted to $G$-invariant subspace blocks.
	
\begin{lemma}\label{lem:entropy}
	Let $q>0$, let $\mu$ be a nonzero finite $G$-invariant Borel measure on $\Sph$, and assume that $\mu$ satisfies the $G$-invariant $q$-th subspace mass inequality \eqref{G qth subspace mass}. Fix $\varepsilon_0>0$, and let $d_1,\dots,d_m$ be positive integers with
	\[
	d_1+\cdots+d_m=n.
	\]
	For each $\ell\in\mathbb N$, let
	\[
	Q_\ell
	=
	\left\{
	x\in \R^n:
	\sum_{j=1}^m \frac{\abs{P_{V_{j,\ell}}x}^2}{b_{j,\ell}^2}\le 1
	\right\},
	\]
	where
	\[
	\R^n=V_{1,\ell}\oplus\cdots\oplus V_{m,\ell}
	\]
	is an orthogonal decomposition into $G$-invariant subspaces such that
	\[
	\dim V_{j,\ell}=d_j,\qquad j=1,\dots,m,
	\]
	and
	\[
	0<b_{1,\ell}\le \cdots \le b_{m,\ell},\qquad b_{m,\ell}\ge \varepsilon_0.
	\]
Assume that, for each $j$,
$V_{j,\ell}\to V_j$ in $\Gr(n,d_j)$ as $\ell\to\infty$. Set
	\[
	D_r=d_1+\cdots+d_r,\qquad D_0=0.
	\]
	If $q\le n$, let $r$ be the unique index such that
	\[
	D_{r-1}<q\le D_r.
	\]
	If $q>n$, set $r=m$. Then there exist positive constants $t_0,\ell_0$ and a constant $C$ such that, for every $\ell\ge \ell_0$,
	\begin{equation}\label{entropy estimate}
	E_\mu(Q_\ell)
	\le
	-\sum_{j=1}^{r-1}\frac{d_j}{q}\log b_{j,\ell}
	-\left(1-\frac{D_{r-1}}{q}\right)\log b_{r,\ell}
	+t_0\log b_{1,\ell}+C.
	\end{equation}
\end{lemma}
	
	\begin{proof}
		By Lemma~\ref{lem:Grass}, each $V_j$ is $G$-invariant, the subspaces $V_1,\dots,V_m$ are pairwise orthogonal, and
		\[
		\R^n=V_1\oplus\cdots\oplus V_m.
		\]
Assume first that $m\ge 2$.	Let $\{\Omega_{j,\delta}\}_{j=1}^m$ be the spherical partition as in \eqref{partation}, namely,
	\[
	\Omega_{j,\delta}
	=
	\left\{
	u\in \Sph:
	\abs{P_{V_j}u}\ge \delta
	\text{ and }
	\abs{P_{V_t}u}<\delta\text{ for all }t>j
	\right\},
	\qquad j=1,\dots,m.
	\] For each $r'=1,\dots,m$, set
		\[
		W_{r'}=V_1\oplus\cdots\oplus V_{r'}.
		\]
		Since each $W_{r'}$ is a proper $G$-invariant subspace whenever $r'<m$, the hypothesis and \eqref{eq:Omega-measure-limit} yield
		\[
		\lim_{\delta\to 0^{+}}\frac{\mu(\Omega_{1,\delta}\cup\cdots\cup\Omega_{r',\delta})}{\abs{\mu}}
		=
		\frac{\mu(W_{r'}\cap \Sph)}{\abs{\mu}}
		<
		\min\left\{\frac{D_{r'}}{q},1\right\},
		\qquad r'=1,\dots,m-1.
		\]
		Hence, there exist $t_0>0$ and $\delta_0>0$ such that
		\begin{equation}\label{sigmar}
		\frac{\mu(\Omega_{1,\delta_0}\cup\cdots\cup\Omega_{r',\delta_0})}{\abs{\mu}}
		<
		\min\left\{\frac{D_{r'}}{q},1\right\}-t_0
		=:\sigma_{r'}
		\end{equation}
		for every $r'=1,\dots,m-1$. Set $\sigma_0=0$ and $\sigma_m=1$.

		The convergence $V_{j,\ell}\to V_j$ gives an $\ell_0\in\mathbb N$ such that,
		for every $\ell\ge \ell_0$,
		\[
		\|P_{V_{j,\ell}}-P_{V_j}\|<\delta_0/2
		\qquad\text{for every }j=1,\dots,m.
		\]
		If $u\in \Omega_{j,\delta_0}$, then $\abs{P_{V_j}u}\ge \delta_0$, hence
		\[
		\abs{P_{V_{j,\ell}}u}
		\ge
		\abs{P_{V_j}u}-\abs{(P_{V_{j,\ell}}-P_{V_j})u}
		\ge
		\delta_0/2.
		\]
		Since
		\[
		h_{Q_\ell}(u)^2
		=
		\sum_{j=1}^m b_{j,\ell}^2\abs{P_{V_{j,\ell}}u}^2,
		\]
		it follows that
		\begin{equation}\label{support1}
		h_{Q_\ell}(u)\ge b_{j,\ell}\delta_0/2
		\qquad\text{for }u\in \Omega_{j,\delta_0}.
		\end{equation}
		Define
		\[
		\lambda_{j,\delta_0}=\frac{\mu(\Omega_{j,\delta_0})}{\abs{\mu}}.
		\]
		Then by \eqref{sigmar},
		\begin{equation}
		\lambda_{1,\delta_0}+\cdots+\lambda_{m,\delta_0}=1,
		\qquad
		\lambda_{1,\delta_0}+\cdots+\lambda_{r',\delta_0}< \sigma_{r'}
		\quad (r'=1,\dots,m-1).
		\end{equation}
		Moreover, by \eqref{support1}
	\begin{align}
		E_\mu(Q_\ell)
		&=
		-\frac1{\abs{\mu}}\int_{\Sph}\log h_{Q_\ell}\,d\mu \notag\\
		&=
		-\frac1{\abs{\mu}}\sum_{j=1}^m \int_{\Omega_{j,\delta_0}}\log h_{Q_\ell}\,d\mu \notag\\
		&\le
		-\log(\delta_0/2)-\sum_{j=1}^m \lambda_{j,\delta_0}\log b_{j,\ell}.
		\label{entropy1}
	\end{align}
		Applying Lemma~\ref{lem:weights} with $a_j=\log b_{j,\ell}$, we obtain
		\begin{equation}\label{difference}
		\sum_{j=1}^m \lambda_{j,\delta_0}\log b_{j,\ell}
		\ge
		\sum_{j=1}^m (\sigma_j-\sigma_{j-1})\log b_{j,\ell}.
		\end{equation}
		By the definition of the $\sigma_j$ in \eqref{sigmar}, the quantities $\sigma_j-\sigma_{j-1}$ depend on the position of $r$.
		
		If $r=1$, that is, $0<q\le D_1$, then
		\[
		\sigma_1-\sigma_0=1-t_0,
		\]
		\[
		\sigma_j-\sigma_{j-1}=0\qquad (2\le j\le m-1),
		\]
		and
		\[
		\sigma_m-\sigma_{m-1}=t_0.
		\]
		
		If $2\le r\le m-1$, then
		\[
		\sigma_1-\sigma_0=\frac{d_1}{q}-t_0,
		\]
		\[
		\sigma_j-\sigma_{j-1}=\frac{d_j}{q}\qquad (2\le j\le r-1),
		\]
		\[
		\sigma_r-\sigma_{r-1}=1-\frac{D_{r-1}}{q},
		\]
		\[
		\sigma_j-\sigma_{j-1}=0\qquad (r+1\le j\le m-1),
		\]
		and
		\[
		\sigma_m-\sigma_{m-1}=t_0.
		\]
		
		If $r=m$, that is, $q>D_{m-1}$, then
		\[
		\sigma_1-\sigma_0=\frac{d_1}{q}-t_0,
		\]
		\[
		\sigma_j-\sigma_{j-1}=\frac{d_j}{q}\qquad (2\le j\le m-1),
		\]
		and
		\[
		\sigma_m-\sigma_{m-1}=1-\frac{D_{m-1}}{q}+t_0.
		\]

	Therefore, by \eqref{difference}, if $2\le r\le m-1$, then
	\begin{align}
		\sum_{j=1}^m \lambda_{j,\delta_0}\log b_{j,\ell}
		&\ge
		\left(\frac{d_1}{q}-t_0\right)\log b_{1,\ell}
		+
		\sum_{j=2}^{r-1}\frac{d_j}{q}\log b_{j,\ell}
		+
		\left(1-\frac{D_{r-1}}{q}\right)\log b_{r,\ell}
		+
		t_0\log b_{m,\ell} \notag\\
		&=
		\sum_{j=1}^{r-1}\frac{d_j}{q}\log b_{j,\ell}
		+
		\left(1-\frac{D_{r-1}}{q}\right)\log b_{r,\ell}
		-
		t_0\log b_{1,\ell}
		+
		t_0\log b_{m,\ell}.
		\label{lower-estimate}
	\end{align}
	
	The same estimate also holds when $r=1$ or $r=m$, with the convention  that the sum
	\[
	\sum_{j=1}^{r-1}\frac{d_j}{q}\log b_{j,\ell}
	\]
	is empty when $r=1$. Indeed, if $r=1$, then \eqref{difference} gives
	\[
	\sum_{j=1}^m \lambda_{j,\delta_0}\log b_{j,\ell}
	\ge
	(1-t_0)\log b_{1,\ell}+t_0\log b_{m,\ell}
	=
	\left(1-\frac{D_0}{q}\right)\log b_{1,\ell}
	-t_0\log b_{1,\ell}
	+t_0\log b_{m,\ell},
	\]
	which is exactly \eqref{lower-estimate}. If $r=m$, then \eqref{difference} yields
	\begin{align*}
		\sum_{j=1}^m \lambda_{j,\delta_0}\log b_{j,\ell}
		&\ge
		\left(\frac{d_1}{q}-t_0\right)\log b_{1,\ell}
		+
		\sum_{j=2}^{m-1}\frac{d_j}{q}\log b_{j,\ell}
		+
		\left(1-\frac{D_{m-1}}{q}+t_0\right)\log b_{m,\ell}\\
		&=
		\sum_{j=1}^{m-1}\frac{d_j}{q}\log b_{j,\ell}
		+
		\left(1-\frac{D_{m-1}}{q}\right)\log b_{m,\ell}
		-
		t_0\log b_{1,\ell}
		+
		t_0\log b_{m,\ell},
	\end{align*}
	which is again \eqref{lower-estimate}.
	
	Combining \eqref{entropy1}, \eqref{lower-estimate},  and using $b_{m,\ell}\ge \varepsilon_0$, we conclude that
		\[
		E_\mu(Q_\ell)
		\le
		-\sum_{j=1}^{r-1}\frac{d_j}{q}\log b_{j,\ell}
		-\left(1-\frac{D_{r-1}}{q}\right)\log b_{r,\ell}
		+t_0\log b_{1,\ell}+C,
		\]
		where $C$ depends only on $\delta_0$, $t_0$, and $\varepsilon_0$.
		
		If $m=1$, then $V_{1,\ell}=\R^n$ and
	$
		Q_\ell=b_{1,\ell}B^n.
		$
		Hence
		$
		h_{Q_\ell}(u)=b_{1,\ell}$
		for all $u\in \Sph$,
		and therefore
		\[
		E_\mu(Q_\ell)=-\log b_{1,\ell}.
		\]
		In this case $r=1$. Fix any $t_0>0$ and take $\ell_0=1$. Since
		$b_{1,\ell}\ge\varepsilon_0$, choosing $C\ge -t_0\log\varepsilon_0$ gives
		\[
		-\log b_{1,\ell}
		\le
		-\log b_{1,\ell}+t_0\log b_{1,\ell}+C,
		\]
		which is exactly \eqref{entropy estimate}.
	\end{proof}

	\section{Estimates for dual quermassintegrals}
In variational approaches to the dual Minkowski problem, estimates for dual quermassintegrals are needed. A classical way to obtain such estimates is to compare with suitably chosen barrier bodies; see, for example, \cite{BLYZZ2019,HLYZdual,zhao2018}. In this section, we estimate appropriate barrier bodies adapted to the $G$-invariant subspace blocks.

	\begin{lemma}\label{lem:ballblock}
		Let $d$ and $m$ be integers with $d\ge 1$ and $m\ge 0$, and let
		$0<\alpha<d$. Then there exists a constant $C=C(d,m,\alpha)$ such that
		\[
		\widetilde W_{(d+m)-\alpha}\bigl(bB^d\times B^m\bigr)\le C\,b^\alpha
		\qquad\text{for every }0<b\le 1.
		\]
	\end{lemma}
	
	\begin{proof}
		We first assume $m>0$. By Lemma~\ref{lem:gauss} and \eqref{product radial},
		\[
		c_0(d+m,\alpha)\,\widetilde W_{(d+m)-\alpha}(bB^d\times B^m)
		=
		\int_{\R^d\times \R^m}
		\min\left\{\frac{b}{\abs{x}},\frac1{\abs{y}}\right\}^\alpha
		e^{-\abs{x}^2-\abs{y}^2}\,dx\,dy.
		\]
		Split the domain into
		\[
		A=\{(x,y): \abs{x}\ge b\abs{y}\},
		\qquad
		B=\{(x,y): \abs{x}< b\abs{y}\}.
		\]
		On $A$ the minimum equals $b/\abs{x}$, and on $B$ it equals $1/\abs{y}$.
		
		For the $A$-integral, using polar coordinates in $\R^d$ and $\R^m$,
\begin{align*}
	I_A
	&=
	\int_A \left(\frac{b}{\abs{x}}\right)^\alpha e^{-\abs{x}^2-\abs{y}^2}\,dx\,dy\\
	&=
	\int_{\R^m}\int_{\{x\in\R^d:\,\abs{x}\ge b\abs{y}\}}
	\left(\frac{b}{\abs{x}}\right)^\alpha e^{-\abs{x}^2-\abs{y}^2}\,dx\,dy\\
	&=
	d\omega_d b^\alpha
	\int_{\R^m}\int_{b\abs{y}}^\infty r^{d-\alpha-1}e^{-r^2-\abs{y}^2}\,dr\,dy\\
	&=
	dm\omega_d\omega_m b^\alpha
	\int_0^\infty\int_{bs}^\infty r^{d-\alpha-1}s^{m-1}e^{-r^2-s^2}\,dr\,ds\\
	&\le
	dm\omega_d\omega_m b^\alpha
	\int_0^\infty\int_0^\infty r^{d-\alpha-1}s^{m-1}e^{-r^2-s^2}\,dr\,ds\\
	&\le C_{1}\,b^\alpha,
\end{align*}
		because $d-\alpha>0$ and $m>0$.

		For the $B$-integral,
		\begin{align*}
			I_B
			&=
			\int_B \abs{y}^{-\alpha}e^{-\abs{x}^2-\abs{y}^2}\,dx\,dy\\
			&=
			d\omega_d
			\int_{\R^m}\int_0^{b\abs{y}}\abs{y}^{-\alpha}e^{-\abs{y}^2-r^{2}}
		 r^{d-1}\,dr\,dy\\
			&=
			dm\omega_d\omega_m
			\int_0^\infty s^{m-\alpha-1}e^{-s^2}
			\left(\int_0^{bs} r^{d-1}e^{-r^2}\,dr\right)\,ds\\
			&\le
			dm\omega_d\omega_m
			\int_0^\infty s^{m-\alpha-1}e^{-s^2}
			\left(\int_0^{bs} r^{d-1}\,dr\right)\,ds\\
			&=
			m\omega_d\omega_m\,b^d
			\int_0^\infty s^{m+d-\alpha-1}e^{-s^2}\,ds\\
			&\le C_{2}\,b^d.
		\end{align*}
		Since $m+d-\alpha>0$, the last integral is finite. Moreover $0<b\le 1$ and $d>\alpha$, hence $b^d\le b^\alpha$. Therefore $I_B\le C_{2}b^\alpha$. This proves the claim when $m>0$.
		
	If $m=0$, the conclusion follows immediately from the definition of the dual quermassintegral.
	\end{proof}
	
We need the following upper bound for the dual quermassintegral of a Cartesian product.
	
		\begin{lemma}[Lemma 5.3, \cite{BLYZZ2019}]\label{lem:product}
		Let $1\le k<q<n$. Then, for each $K\in \mathcal K_o^{k}$ and $L\in \mathcal K_o^{n-k}$,
		\[
		\widetilde W_{n-q}(K\times L)
		\le
		C(n,k,q)\,V_k(K)\,\widetilde W_{(n-k)-(q-k)}^{\,(n-k)}(L).
		\]
	\end{lemma}
	
	The following lemma will be needed in the proof of the existence theorem.
	
	\begin{lemma}\label{lem:blockbarrier}
		Let
		\[
		\R^n=U\oplus V\oplus W
		\]
		be an orthogonal decomposition with $\dim U=k\ge 0$, $\dim V=d\ge 1$, and $\dim W=m\ge 0$. Let $E\subset U$ be a $k$-dimensional origin-centered ellipsoid. Suppose that
		\[
		k<q<k+d.
		\]
		Then there exists a constant $C=C(n,k,d,q)$ such that, for every $0<b\le 1$,
		\[
		\widetilde W_{n-q}\bigl(E\times bB_V\times B_W\bigr)
		\le
		C\,V_k(E)\,b^{\,q-k}.
		\]
	\end{lemma}
	
	\begin{proof}
		Set $\alpha=q-k$. Then $0<\alpha<d$. If $k=0$, the claim is exactly Lemma~\ref{lem:ballblock}. Assume now that $k\ge 1$. Applying Lemma~\ref{lem:product} with $K=E$ and $L=bB_V\times B_W$, we obtain
		\[
		\widetilde W_{n-q}(E\times bB_V\times B_W)
		\le
		C_1(n,k,q)\,V_k(E)\,\widetilde W_{(d+m)-\alpha}(bB_V\times B_W).
		\]
		By Lemma~\ref{lem:ballblock},
		\[
		\widetilde W_{(d+m)-\alpha}(bB_V\times B_W)\le C_2(d,m,\alpha)\,b^\alpha.
		\]
		Combining the two estimates yields the result.
	\end{proof}
	
	When $q>n$, we need the following estimate.
	
	\begin{lemma}\label{lem:W-q>n}
		Let $q>n$, and let
		\[
		Q=
		\left\{
		x\in \R^n:
		\sum_{i=1}^m \frac{\abs{P_{V_i}x}^2}{b_i^2}\le 1
		\right\},
		\]
		where
		\[
		\R^n=V_1\oplus\cdots\oplus V_m
		\]
		is an orthogonal decomposition, $\dim V_i=d_i$, and
		\[
		0<b_1\le \cdots \le b_m.
		\]
		Then
		\[
		\widetilde W_{n-q}(Q)\le C(n,q)\,b_m^{\,q-n}\prod_{i=1}^m b_i^{d_i}.
		\]
	\end{lemma}
	
	\begin{proof}
		Let $T$ be the linear operator defined by
		\[
		T|_{V_i}=b_i\,I_{V_i},
		\qquad i=1,\dots,m.
		\]
		Then $Q=T(B^n)$, $\det T=\prod_{i=1}^m b_i^{d_i}$, and $\|T\|=b_m$.
		
		Since $q>0$, polar coordinates give
		\[
		\widetilde W_{n-q}(Q)=\frac{q}{n}\int_Q \abs{x}^{\,q-n}\,dx.
		\]
		Using the change of variables $x=Ty$, we get
		\begin{equation}\label{Ttrans}
		\widetilde W_{n-q}(Q)
		=
		\frac{q}{n}\det(T)\int_{B^n} \abs{Ty}^{\,q-n}\,dy.
		\end{equation}
		Because $q-n>0$, by \eqref{normdef}, for every $y\ne0$,
		\[
		\abs{Ty}^{\,q-n}=\left|T\left(\frac{y}{\abs y}\right)\right|^{q-n}\abs y^{\,q-n}\le \|T\|^{q-n}\abs{y}^{\,q-n}=b_m^{\,q-n}\abs{y}^{\,q-n}.
		\]
		Therefore by \eqref{Ttrans},
		\[
		\widetilde W_{n-q}(Q)
		\le
		\frac{q}{n}\det(T)\,b_m^{\,q-n}\int_{B^n}\abs{y}^{\,q-n}\,dy.
		\]
		Since $\int_{B^n}\abs{y}^{\,q-n}\,dy<\infty$, the desired estimate follows.
	\end{proof}

	\vspace{1em}
	\section{Existence for the $G$-invariant dual Minkowski problem}
	
	Let $\mu$ be a nonzero finite Borel measure on $\Sph$. The entropy functional of $\mu$, $E_\mu:C^+(\Sph)\to \mathbb R$, is defined by
	\[
	E_\mu(f)
	=
	-\frac{1}{|\mu|}
	\int_{\Sph}\log f(v)\,d\mu(v),
	\qquad f\in C^+(\Sph).
	\]
	For $q>0$, define the variational functional
	\begin{equation}\label{varifun}
		\Phi_\mu(f)
		=
		E_\mu(f)
		+
		\frac1q\log \widetilde W_{n-q}([f]),
		\qquad f\in C^+(\Sph),
	\end{equation}
	where $[f]$ is the Wulff shape generated by $f$. 	Since $\widetilde W_{n-q}$ is homogeneous of degree $q$, we have
	\begin{equation}\label{hom}
		\Phi_\mu(cf)=\Phi_\mu(f)\qquad\text{for every }c>0.
	\end{equation}
By the definition of the Wulff shape, $0<h_{[f]}\le f$ and $[h_{[f]}]=[f]$. Hence
\begin{equation}\label{wulffphi}
	\Phi_{\mu}(f)\le \Phi_{\mu}(h_{[f]}).
\end{equation}
		
	We need  the following lemma  about group-invariant measures from \cite{BKMZ group}:
	
	\begin{lemma}[Lemma 5.1, \cite{BKMZ group}]\label{unique}
		Let $G$ be a closed subgroup of $\On$. If $\mu_1$ and $\mu_2$ are $G$-invariant finite Borel measures on $S^{n-1}$, then $\mu_1 = \mu_2$ if and only if $\int_{S^{n-1}} g d\mu_1 = \int_{S^{n-1}} g d\mu_2$ for any $G$-invariant continuous function $g : S^{n-1} \to \mathbb{R}$.
	\end{lemma}

	To solve the Minkowski problem via variational methods, we hope to transform the
	 dual Minkowski problem into an optimization problem.
	 
	 Let $C^+_G(\Sph)$ denote the set of all positive continuous $G$-invariant functions on $\Sph$, namely
	 \[
	 C^+_G(\Sph)=\{f\in C^+(\Sph): f \text{ is } G\text{-invariant}\}.
	 \]
	
	\begin{proposition}\label{prop:EL}
		Let $q>0$.  Let $G\subset O(n)$ be a closed subgroup with no nonzero fixed points. Let $\mu$ be a nonzero finite $G$-invariant Borel measure on $\Sph$. Suppose that there exists $K_0\in \KG$ such that
		\[
		\Phi_\mu(h_{K_0})
		=
		\sup\{\Phi_\mu(f): f\in C^+_{G}(\Sph)\}.
		\]
		Then there exists $c>0$ such that
		\[
		\mu=\widetilde{C}_q(cK_0,\cdot).
		\]
	\end{proposition}
	
	\begin{proof}
		Choose $c>0$ so that
		\[
		\widetilde W_{n-q}(cK_0)=\abs{\mu}.
		\]
		Fix $g\in C(\Sph)$ that is $G$-invariant. For small $t$, set
		\[
		h_t(v)=h_{cK_0}(v)e^{t g(v)},\qquad v\in \Sph.
		\]
	Since $K_0\in \KG\subset \mathcal K^{n}_{o}$, the support function $h_{cK_0}$ is positive and continuous on $\Sph$. Since both $h_{cK_0}$ and $g$ are $G$-invariant, so is $h_t$. Hence $h_t\in C_G^+(\Sph)$ for all small $t$. By the homogeneity \eqref{hom}, $h_{cK_0}=c h_{K_0}$ is also a maximizer of $\Phi_\mu$ over $C_G^+(\Sph)$. Hence
		\[
		\left.\frac{d}{dt}\Phi_\mu(h_t)\right|_{t=0}=0.
		\]
		Using the logarithmic Wulff variation formula \eqref{vari} and the normalization $\widetilde W_{n-q}(cK_0)=\abs{\mu}$, we obtain
		\[
		0
		=
		-\frac1{\abs{\mu}}\int_{\Sph}g\,d\mu
		+
		\frac1{\abs{\mu}}
		\int_{\Sph} g\,d\widetilde{C}_q(cK_0,\cdot),
		\]
		hence
		\[
		\int_{\Sph} g\,d\mu
		=
		\int_{\Sph} g\,d\widetilde{C}_q(cK_0,\cdot)
		\qquad
		\text{for every }g\in C(\Sph)\text{ that is }G\text{-invariant}.
		\]
	Since $cK_0$ is $G$-invariant, \cite[Theorem 6.8]{lyz2018} implies that
	$\widetilde{C}_q(cK_0,\cdot)$ is also $G$-invariant. By Lemma~\ref{unique}, we obtain
		$\mu=\widetilde{C}_q(cK_0,\cdot)$.
	\end{proof}
	
	We are now ready for the existence proof.

	\begin{theorem}\label{thm:q>0}
		Let $q>0$. Let $G\subset O(n)$ be a closed subgroup with no nonzero fixed points. Let $\mu$ be a nonzero finite $G$-invariant Borel measure on $\Sph$. Assume that $\mu$ satisfies the $G$-invariant $q$-th subspace mass inequality.  Then there exists  $K\in \KG$ such that
		\[
		\mu=\widetilde{C}_q(K,\cdot).
		\]
	\end{theorem}
	
	\begin{proof}
	By \eqref{wulffphi}, we can assume that $\{K_\ell\}\subset \KG$ is a maximizing sequence such that
		\[
		\lim_{\ell \to \infty}\Phi_\mu(h_{K_\ell})= \sup\{\Phi_\mu(f): f\in C^+_{G}(\Sph)\}.
		\]
		Because $\Phi_\mu$ is homogeneous of degree $0$, we may assume
		\[
		\diam(K_\ell)=1
		\qquad\text{for every }\ell.
		\]
		Since $K_\ell\in\KG\subset\mathcal K_o^n$, we have $0\in \operatorname{int}K_\ell$ for every $\ell$.
		By Blaschke's selection theorem, after passing to a subsequence we may assume that $K_\ell$ converges in the Hausdorff metric to a compact convex set $K_0$. Since $0\in K_\ell$ for every $\ell$, we have $0\in K_0$. Since $K_\ell$ is $G$-invariant for every $\ell$, the limit $K_0$ is also $G$-invariant. If $K_0$ has nonempty interior, then $K_0\in \KG$, and the continuity of $\Phi_\mu$  with respect to the Hausdorff metric shows that $h_{K_0}$ attains the supremum. Proposition~\ref{prop:EL} then yields the desired solution. It therefore remains to exclude the possibility that $K_0$ is contained in a proper subspace of $\R^n$.
		
		Assume, for contradiction, that $K_0$ is contained in a proper subspace. Let $Q_\ell$ be the John ellipsoid of $K_\ell$. By Lemma~\ref{lem:John}, each $Q_\ell$ is $G$-invariant and centered at the origin. Then,
		as is well known (see Schneider \cite{schneiderbook2014}, p.~588),
		\begin{equation}\label{john inclusion}
		Q_\ell\subset K_\ell\subset nQ_\ell.
		\end{equation}
		Writing $Q_\ell$ in spectral-block form as in Lemma~\ref{lem:John}, we have
		\[
		Q_\ell=
		\left\{
		x\in \R^n:
		\sum_{j=1}^{m_\ell}\frac{\abs{P_{V_{j,\ell}}x}^2}{b_{j,\ell}^2}\le 1
		\right\},
		\]
		where the $V_{j,\ell}$ are pairwise orthogonal $G$-invariant eigenspaces and
		\[
		0<b_{1,\ell}<\cdots<b_{m_\ell,\ell}.
		\]
		Since $m_\ell\in\{1,\dots,n\}$ and the multiplicity pattern of the distinct eigenvalues is one of finitely many partitions of $n$, by passing to a further subsequence we may assume that $m_\ell\equiv m$ is independent of $\ell$ and that
		\[
		\dim V_{j,\ell}=d_j
		\qquad\text{for }j=1,\dots,m.
		\]
		Since each $\Gr(n,d_j)$ is compact, we may pass to a further subsequence and assume that, for every $j$,
		\[
		V_{j,\ell}\to V_j \quad \text{in } \Gr(n,d_j).
		\]
		By Lemma~\ref{lem:Grass},
		\[
		\R^n=V_1\oplus\cdots\oplus V_m
		\]
		is an orthogonal decomposition into $G$-invariant subspaces.
		
		Since
		\[
		Q_\ell\subset K_\ell,\qquad \diam(K_\ell)=1,\qquad K_\ell\subset nQ_\ell,
		\]
		the largest semi-axis satisfies
		\begin{equation}\label{axis ine}
			\frac1{2n}\le b_{m,\ell}\le \frac12.
		\end{equation}
		We also pass to a further subsequence so that each $b_{j,\ell}$ converges to some $b_j\in [0,\infty)$. In particular, the hypothesis of Lemma~\ref{lem:entropy} holds with $\varepsilon_0=1/(2n)$.
		
We claim that
\[
b_{1,\ell}\to b_{1}=0.
\]
Indeed, if $b_1>0$, then for all sufficiently large
$\ell$ one has $(b_1/2)B^n\subset Q_\ell\subset K_\ell$. Passing to the limit gives $(b_1/2)B^n\subset K_0$, contradicting the fact that
$K_0$ has empty interior. Hence $b_1=0$.

		Let
		\[
		D_r=d_1+\cdots+d_r,\qquad D_0=0.
		\]
		If $q\le n$, let $r$ be the unique index such that
		\[
		D_{r-1}<q\le D_r.
		\]
		If $q>n$, set $r=m$.
		Applying Lemma~\ref{lem:entropy} and \eqref{john inclusion}, we obtain
		\begin{equation}\label{eq:entropy-main}
			E_\mu(K_\ell)
			\le
			E_\mu(Q_\ell)
			\le
			-\sum_{j=1}^{r-1}\frac{d_j}{q}\log b_{j,\ell}
			-\left(1-\frac{D_{r-1}}{q}\right)\log b_{r,\ell}
			+t_0\log b_{1,\ell}+C
		\end{equation}
		for all sufficiently large $\ell$.
		
	We now estimate the dual quermassintegral term. We distinguish four cases.
		
		\medskip
		\noindent\textit{Case 1.}  $D_{r-1}<q<D_r$.
		Set
		\[
		U_\ell=V_{1,\ell}\oplus\cdots\oplus V_{r-1,\ell},
		\qquad
		V_\ell=V_{r,\ell},
		\qquad
		W_\ell=(U_\ell\oplus V_\ell)^\perp.
		\]
		Let $E_\ell\subset U_\ell$ be the ellipsoid
		\[
		E_\ell=
		\left\{
		x\in U_\ell:
		\sum_{j=1}^{r-1}\frac{\abs{P_{V_{j,\ell}}x}^2}{b_{j,\ell}^2}\le 1
		\right\},
		\]
		with the convention that $E_\ell$ is absent if $r=1$. Define
		\[
		G_\ell=E_\ell\times b_{r,\ell}B_{V_\ell}\times B_{W_\ell}.
		\]
		Because $b_{1,\ell}< \cdots< b_{m,\ell}\le 1/2$, it follows directly from the definitions  that
		\[
		Q_\ell\subset G_\ell,
		\]
		hence
		\[
		K_\ell\subset nQ_\ell\subset nG_\ell.
		\]
		Now $\dim U_\ell=D_{r-1}$ and $\dim V_\ell=d_r$, and
		\[
		D_{r-1}<q<D_{r}=D_{r-1}+d_r,
		\]
		so Lemma~\ref{lem:blockbarrier},  applied with $k=D_{r-1}$ and $d=d_r$, gives
		\[
		\widetilde W_{n-q}(G_\ell)
		\le
		C_1\,V_{D_{r-1}}(E_\ell)\,b_{r,\ell}^{\,q-D_{r-1}}.
		\]
		Since
		\[
		V_{D_{r-1}}(E_\ell)=C_{2}\,\prod_{j=1}^{r-1} b_{j,\ell}^{d_j},
		\]
		we obtain
		\[
		\frac1q\log \widetilde W_{n-q}(K_\ell)
		\le
		\frac1q\log \widetilde W_{n-q}(nG_\ell)
		\le
		\sum_{j=1}^{r-1}\frac{d_j}{q}\log b_{j,\ell}
		+
		\left(1-\frac{D_{r-1}}{q}\right)\log b_{r,\ell}
		+
		C_3.
		\]
		Combining this with \eqref{eq:entropy-main} and \eqref{varifun}, we obtain
		\[
		\Phi_\mu(h_{K_\ell})\le t_0\log b_{1,\ell}+C_4.
		\]
		Since $t_0>0$ and  $b_{1,\ell}\to 0$, it follows that
		\[
		\Phi_\mu(h_{K_\ell})\to -\infty,
		\]
		contradicting the fact that $\Phi_\mu(h_{K_\ell})\to \sup\{\Phi_\mu(f):f\in C^+_{G}(\Sph)\}\ge\Phi_\mu(h_{B^{n}})>-\infty$.
		
	\medskip
	\noindent\textit{Case 2.}  $q=D_r$ for some $0<r<m$. Choose $q_0$ such that
		\[
		q<q_0<D_{r+1}
		\]
		and so close to $q$ that
		\begin{equation}\label{q00}
		t_0+q\left(\frac1{q_0}-\frac1q\right)>0.
		\end{equation}
		By the monotonicity of $L^p$-norms on the sphere,
		\[
		\left(\frac1{n\omega_n}\int_{\Sph}\rho_{K_\ell}(u)^q\,du\right)^{1/q}
		\le
		\left(\frac1{n\omega_n}\int_{\Sph}\rho_{K_\ell}(u)^{q_0}\,du\right)^{1/q_0},
		\]
		hence
		\begin{equation}\label{q0}
		\frac1q\log \widetilde W_{n-q}(K_\ell)
		\le
		\frac1{q_0}\log \widetilde W_{n-q_0}(K_\ell)+C_1.
		\end{equation}
		Set
		\[
		U_\ell=V_{1,\ell}\oplus\cdots\oplus V_{r,\ell},
		\qquad
		V_\ell=V_{r+1,\ell},
		\qquad
		W_\ell=(U_\ell\oplus V_\ell)^\perp,
		\]
		and let $E_\ell\subset U_\ell$ be the ellipsoid
		\[
		E_\ell=
		\left\{
		x\in U_\ell:
		\sum_{j=1}^{r}\frac{\abs{P_{V_{j,\ell}}x}^2}{b_{j,\ell}^2}\le 1
		\right\}.
		\]
		Define
		\[
		H_\ell=E_\ell\times b_{r+1,\ell}B_{V_\ell}\times B_{W_\ell}.
		\]
		Again $Q_\ell\subset H_\ell$, so $K_\ell\subset nH_\ell$. We have
		\[
		\dim U_\ell=D_r=q
		\qquad\text{and}\qquad
		q<q_0<D_{r+1}=q+d_{r+1}.
		\]
		Applying Lemma~\ref{lem:blockbarrier} to $H_\ell$, we get
		\[
		\widetilde W_{n-q_0}(H_\ell)
		\le
		C_2\,V_q(E_\ell)\,b_{r+1,\ell}^{\,q_0-q}.
		\]
		Since
		\[
		V_q(E_\ell)=C_{3}\prod_{j=1}^{r}b_{j,\ell}^{d_j},
		\]
		we obtain
		\[
		\frac1{q_0}\log \widetilde W_{n-q_0}(K_\ell)
		\le
		\frac1{q_0}\log \widetilde W_{n-q_0}(nH_\ell)
		\le
		\sum_{j=1}^{r}\frac{d_j}{q_0}\log b_{j,\ell}
		+
		\frac{q_0-q}{q_0}\log b_{r+1,\ell}
		+
		C_4.
		\]
		Since $b_{r+1,\ell}\le 1/2$, the second logarithmic term is nonpositive, and by \eqref{q0},
		\[
		\frac1q\log \widetilde W_{n-q}(K_\ell)
		\le
		\sum_{j=1}^{r}\frac{d_j}{q_0}\log b_{j,\ell}+C_5.
		\]
		Now \eqref{eq:entropy-main} becomes
		\[
		E_\mu(K_\ell)
		\le
		-\sum_{j=1}^{r}\frac{d_j}{q}\log b_{j,\ell}
		+t_0\log b_{1,\ell}+C,
		\]
		because $q=D_r$. Combining the last two estimates yields
		\[
		\Phi_\mu(h_{K_\ell})
		\le
		\sum_{j=1}^{r} d_j\left(\frac1{q_0}-\frac1q\right)\log b_{j,\ell}
		+t_0\log b_{1,\ell}
		+
		C_6.
		\]
		Since $\frac1{q_0}-\frac1q<0$ and $b_{j,\ell}\ge b_{1,\ell}$, we have
	\[
	\begin{aligned}
		\sum_{j=1}^{r} d_j\left(\frac1{q_0}-\frac1q\right)\log b_{j,\ell}
		&\le
		\left(\sum_{j=1}^{r}d_j\right)
		\left(\frac1{q_0}-\frac1q\right)\log b_{1,\ell} \\
		&=
		q\left(\frac1{q_0}-\frac1q\right)\log b_{1,\ell}.
	\end{aligned}
	\]
		Consequently,
		\[
		\Phi_\mu(h_{K_\ell})
		\le
		\left(
		t_0+q\left(\frac1{q_0}-\frac1q\right)
		\right)\log b_{1,\ell}+C_6.
		\]
		By the choice of $q_0$ in \eqref{q00}, the coefficient of $\log b_{1,\ell}$ is positive, and therefore, since $b_{1,\ell}\to 0$,
		\[
		\Phi_\mu(h_{K_\ell})\to -\infty.
		\]
		This again contradicts the maximizing property of $\{K_\ell\}$.
		
		\medskip
		\noindent\textit{Case 3.}  $q=n=D_{m}$.
		
		In this case, \eqref{eq:entropy-main} becomes
		\begin{equation}\label{q=n entropy}
		E_\mu(K_\ell)
		\le
		E_\mu(Q_\ell)
		\le
		-\frac1n\sum_{j=1}^m d_j\log b_{j,\ell}
		+t_0\log b_{1,\ell}
		+C.
		\end{equation}
		On the other hand, since $K_\ell\subset nQ_\ell$, we have
		\[
		V_n(K_\ell)\le V_n(nQ_\ell)=n^nV_n(Q_\ell)=n^n\omega_n\prod_{j=1}^m b_{j,\ell}^{d_j},
		\]
		and therefore
	\begin{equation}\label{q=n W}
		\frac1n\log V_n(K_\ell)
		\le
		\frac1n\log V_n(Q_\ell)+C_1
		=
		\frac1n\sum_{j=1}^m d_j\log b_{j,\ell}+C_2.
	\end{equation}
		Combining \eqref{q=n entropy} and \eqref{q=n W}, we obtain
		\[
		\Phi_\mu(h_{K_\ell})\le t_0\log b_{1,\ell}+C_3.
		\]
		Since $b_{1,\ell}\to 0$, it follows that
		\[
		\Phi_\mu(h_{K_\ell})\to -\infty,
		\]
		a contradiction.
		
		\medskip
		\noindent\textit{Case 4.}  $q>n=D_{m}$.
		
		In this case, \eqref{eq:entropy-main} becomes
	\begin{equation}\label{q>n1}
		E_\mu(K_\ell)
		\le
		-\sum_{j=1}^{m-1}\frac{d_j}{q}\log b_{j,\ell}
		-\left(1-\frac{D_{m-1}}{q}\right)\log b_{m,\ell}
		+t_0\log b_{1,\ell}+C.
	\end{equation}
		On the other hand, by Lemma~\ref{lem:W-q>n} and the inclusion $K_\ell\subset nQ_\ell$, we have
		\[
		\widetilde W_{n-q}(K_\ell)
		\le
		\widetilde W_{n-q}(nQ_\ell)
		\le
		C_1\, b_{m,\ell}^{\,q-n}\prod_{j=1}^m b_{j,\ell}^{d_j}.
		\]
		Therefore,
	\begin{equation}\label{q>n2}
		\begin{aligned}
			\frac1q\log \widetilde W_{n-q}(K_\ell)
			&\le
			\sum_{j=1}^{m}\frac{d_j}{q}\log b_{j,\ell}
			+
			\frac{q-n}{q}\log b_{m,\ell}
			+
			C_2\\
			&=
			\sum_{j=1}^{m-1}\frac{d_j}{q}\log b_{j,\ell}
			+
			\left(1-\frac{D_{m-1}}{q}\right)\log b_{m,\ell}
			+
			C_2.
		\end{aligned}
\end{equation}
		Combining \eqref{q>n1} and \eqref{q>n2}, we obtain
		\[
		\Phi_\mu(h_{K_\ell})\le t_0\log b_{1,\ell}+C_3.
		\]
		Since $b_{1,\ell}\to 0$, it follows that
		\[
		\Phi_\mu(h_{K_\ell})\to -\infty,
		\]
		a contradiction.

		In all cases we arrive at a contradiction. Therefore $K_0$ cannot be contained in a proper subspace of $\R^n$. Hence $K_0\in \KG$, and the continuity of $\Phi_\mu$ yields that $h_{K_0}$ attains the supremum. Proposition~\ref{prop:EL} then gives the desired solution.
	\end{proof}

	Theorem~\ref{thm:q>0} gives a sufficient condition for the $G$-invariant dual Minkowski problem when $q>0$. Moreover, when $0<q<n$, it follows from  Theorem~\ref{thm:G-invariant-subspace-mass} that the $G$-invariant $q$-th subspace mass inequality is also necessary. However, when $q\ge n$, this condition is clearly not necessary. We next turn to the case $q=n$, namely, the celebrated logarithmic Minkowski problem.  In analogy with the classical subspace concentration condition, we introduce the $G$-invariant subspace concentration condition.
	
\begin{definition}\label{Def $G$-invariant subspace concentration condition}
	Let $G\subset O(n)$. A nonzero finite Borel measure $\mu$ on $S^{n-1}$ is said to satisfy the \textit{$G$-invariant subspace concentration condition} if
	\begin{equation}\label{G-invariant-subspace-concentration}
		\frac{\mu(L\cap S^{n-1})}{\mu(S^{n-1})}
		\le
		\frac{\dim L}{n}
	\end{equation}
	for each proper $G$-invariant subspace $L\subset \R^n$, and whenever equality holds for some proper $G$-invariant subspace $L$,  there exists a subspace $\widetilde L\subset \R^n$ complementary to $L$ such that $\mu$ is concentrated on
	$
	S^{n-1}\cap (L\cup \widetilde L)$.
\end{definition}
	
As with the $G$-invariant $q$-th subspace mass inequality, this condition only tests the mass distribution on $G$-invariant subspaces. For instance, when $G$ is irreducible, the condition is vacuous.  The following theorem gives an existence result in the $G$-invariant setting; when $G=\{\pm I\}$, it recovers the even logarithmic Minkowski problem.

\begin{theorem}\label{thm:G-log}
 Let $G\subset O(n)$ be a closed subgroup with no nonzero fixed points. Let $\mu$ be a nonzero finite $G$-invariant Borel measure on $\Sph$. Assume that $\mu$ satisfies the $G$-invariant subspace concentration condition.
	Then there exists  $K\in \KG$ such that
	\[
\mu=	V_K.
	\]
\end{theorem}

\begin{proof}
We prove the assertion by induction on $n$. The case $n=1$ is trivial.

If the inequality in \eqref{G-invariant-subspace-concentration}  is strict for every proper $G$-invariant subspace, then the conclusion follows immediately from the case $q=n$ of Theorem~\ref{thm:q>0}. Hence we may assume that there exists a proper $G$-invariant subspace $L\subset \R^n$ such that
	\begin{equation}\label{subspaceL}
	\frac{\mu(L\cap \Sph)}{\abs{\mu}}=\frac{\dim L}{n}.
	\end{equation}
	By hypothesis, there exists a complementary subspace $\widetilde L$ of $L$ such that
	\[
	\supp \mu \subset (L\cup \widetilde L)\cap \Sph.
	\]
	
	Define
	\[
	M=\Span\bigl(\supp\mu\setminus L\bigr).
	\]
	Since $\supp \mu\setminus L\ \subset \widetilde L$, we have $M\subset \widetilde L$, hence
	\[
	L\cap M=\{0\}.
	\]
	We claim that $M$ is $G$-invariant. Indeed, if $u\in \supp\mu\setminus L$ and $g\in G$, then $gu\in \supp\mu$ because $\mu$ is $G$-invariant. Since $L$ is $G$-invariant, $gu\notin L$, for otherwise $u=g^{-1}(gu)\in L$, a contradiction. Hence $gu\in \supp\mu\setminus L$, and therefore $M$ is $G$-invariant.

	Next, we claim that $M$ is complementary to $L$. Since
	$
	\supp\mu\setminus L\subset M$,
	we have
	\[
	\supp\mu\subset (L\cup M)\cap\Sph.
	\]
	Together with $L\cap M=\{0\}$, this gives
	\begin{equation}\label{subspaceM}
		\mu(M\cap \Sph)
		=
		|\mu|-\mu(L\cap \Sph)
		=
		\frac{n-\dim L}{n}\,|\mu|.
	\end{equation}
	If $\dim M<n-\dim L$, then $M$ is a proper $G$-invariant subspace, and the assumed subspace concentration inequality would imply
	\[
	\mu(M\cap \Sph)\le \frac{\dim M}{n}\,\abs{\mu}
	<
	\frac{n-\dim L}{n}\,\abs{\mu},
	\]
	a contradiction. Therefore
	$
	\dim M=n-\dim L$.
	Since $L\cap M=\{0\}$, it follows that
	\[
	\R^n=L\oplus M.
	\]
	
	Let
	\[
	\mu_L=\mu|_{L\cap \Sph},
	\qquad
	\mu_M=\mu|_{M\cap \Sph}.
	\]
	Then by \eqref{subspaceL} and \eqref{subspaceM},
	\[
	\abs{\mu_L}=\frac{\dim L}{n}\,\abs{\mu},
	\qquad
	\abs{\mu_M}=\frac{\dim M}{n}\,\abs{\mu}.
	\]
	We claim that $\mu_L$ satisfies the same $G|_L$-invariant subspace concentration condition on $L$, and likewise for $\mu_M$ on $M$.
	Indeed, let $U\subset L$ be a proper $G|_L$-invariant subspace. Then $U$ is also a proper $G$-invariant subspace of $\R^n$, and hence
	\[
	\mu_L(U\cap \Sph)
	=
	\mu(U\cap \Sph)
	\le
	\frac{\dim U}{n}\,\abs{\mu}
	=
	\frac{\dim U}{\dim L}\,\abs{\mu_L}.
	\]
	Suppose equality holds, i.e.,
	\[
	\mu_L(U\cap \Sph)=\frac{\dim U}{\dim L}\,\abs{\mu_L}.
	\]
	Then necessarily
	\[
	\mu(U\cap \Sph)=\frac{\dim U}{n}\,\abs{\mu}.
	\]
	By the  hypothesis, there exists a complementary subspace $U'$ of $U$ in $\R^n$ such that
	\[
	\supp\mu\subset (U\cup U')\cap \Sph.
	\]
	Intersecting with $L$ gives
	\[
	\supp\mu_L\subset (U\cup (U'\cap L))\cap \Sph.
	\]
	Since $U\subset L$ and $\R^n=U\oplus U'$, every $x\in L$ may be written uniquely as
	\[
	x=u+u',
	\qquad
	u\in U,\ u'\in U'\cap L,
	\]
	so
	\[
	L=U\oplus (U'\cap L).
	\]
	Thus $U'\cap L$ is complementary to $U$ in $L$, and $\mu_L$ satisfies the required condition in $L$. The same argument applies to $\mu_M$ in $M$.
	
Moreover, $\mu_L$ and $\mu_M$ are invariant under the restricted actions $G|_L$ and $G|_M$, respectively. Since $G$ is compact, the restricted groups $G|_L$ and $G|_M$ are closed. Since $\mathrm{Fix}(G)=\{0\}$, the restricted actions on $L$ and on $M$ also have no nonzero fixed points. Therefore, by the induction hypothesis, there exist convex bodies
\[
C\subset L,
\qquad
C'\subset M,
\]
invariant under $G|_L$ and $G|_M$, respectively, such that
\[
V_C^{\,L}=\mu_L,
\qquad
V_{C'}^{\,M}=\mu_M.
\]
Viewed as subsets of $\R^n$, both $C$ and $C'$ are $G$-invariant.

Let
\[
D=\{(x+L^\perp)\cap M^\perp:\ x\in C\}\subset M^\perp,
\]
and
\[
D'=\{(y+M^\perp)\cap L^\perp:\ y\in C'\}\subset L^\perp.
\]
By Lemma~5.2 of \cite{CLZ2019} (see also Lemma~7.2 of \cite{KLYZ2013JAMS}), there exist constants $a>0$ and $a'>0$ such that the convex body
\[
K:=aD+a'D'
\]
satisfies
\[
\mu=V_K.
\]

It remains to show that $K$ is $G$-invariant. Since $L$ and $M$ are $G$-invariant and $G\subset O(n)$, the orthogonal complements $L^\perp$ and $M^\perp$ are also $G$-invariant. Moreover, both $C$ and $C'$ are $G$-invariant. Hence, by definition, $D$ and $D'$ are also $G$-invariant. Consequently, $K=aD+a'D'$ is $G$-invariant.
Thus $K\in \KG$.
\end{proof}
\vspace{0.5em}
\begin{remark}
	In fact, as shown in the proof of  Theorem \ref{thm:G-log}, since \(M\subset \widetilde L\) and \(\dim M=\dim \widetilde L\), we have
	$
	M=\widetilde L$.
	That is, the $G$-invariant subspace \(M\) is precisely the complementary subspace \(\widetilde L\).
	
In other words, for a $G$-invariant measure satisfying the condition in Definition~\ref{Def $G$-invariant subspace concentration condition}, the complementary subspace $\widetilde L$ in the equality case is automatically $G$-invariant. Consequently, one may equivalently require the complementary subspace in the equality case to be $G$-invariant.
\end{remark}

So far, we have obtained existence results for the $G$-invariant dual Minkowski problem when $G$ is closed. We now summarize these results and extend them to general subgroups, without assuming closedness. First, we need the following lemma. 
\begin{lemma}\label{lem:G-closure-invariant-subspace}
	Let $G$ be a subgroup of $O(n)$, and let $\overline{G}$ denote the closure of $G$ in $O(n)$. Let $L$ be a linear subspace of $\mathbb R^n$. Then $L$ is $G$-invariant if and only if it is $\overline{G}$-invariant.
\end{lemma}

\begin{proof}
	If $L$ is $\overline{G}$-invariant, then it is clearly $G$-invariant, since $G\subset \overline{G}$.
	
	Conversely, assume that $L$ is $G$-invariant. Since $L$ is a linear subspace of $\mathbb R^n$, it is closed. Let $g\in \overline{G}$ and $x\in L$. Then there exists a sequence $\{g_i\}\subset G$ such that $g_i\to g$. Since $L$ is $G$-invariant, we have $g_i x\in L$ for all $i$. Passing to the limit and using the continuity of the action, we obtain $g_i x\to gx$. As $L$ is closed, it follows that $gx\in L$. Hence $L$ is $\overline{G}$-invariant.
\end{proof}

	Therefore, for general subgroups  of $O(n)$, we can still obtain solutions to the corresponding $G$-invariant  dual Minkowski problem.

\begin{theorem}\label{general exist}
		 
		 Let $q\in(0,n]$, let $G\subset O(n)$ be a subgroup with no nonzero fixed points,
		 and let $\mu$ be a nonzero finite $G$-invariant Borel measure on $\Sph$.
		 Then there exists $K\in\KG$ such that
		 \[
		 \mu=\widetilde C_q(K,\cdot)
		 \]
		if and only if the corresponding condition below holds:
		 \begin{itemize}
		 	\item if $0<q<n$, then $\mu$ satisfies the $G$-invariant $q$-th subspace mass inequality;
		 	\item if $q=n$, then $\mu$ satisfies the $G$-invariant subspace concentration condition.
		 \end{itemize}
\end{theorem}

\begin{proof}
We first prove the sufficiency. Assume that the corresponding condition in the statement holds.	Since $\mu$ is $G$-invariant, by Lemma 7.6 in \cite{shan2026},  $\mu$ is $\overline{G}$-invariant. Since $G$ has no nonzero fixed points and $G\subset \overline{G}$, the group $\overline{G}$ also has no nonzero fixed points. Moreover, by Lemma~\ref{lem:G-closure-invariant-subspace}, the $G$-invariant subspaces are precisely the $\overline G$-invariant subspaces. Hence $\mu$ satisfies the $G$-invariant $q$-th subspace mass inequality when $0<q<n$, and the $G$-invariant subspace concentration condition when $q=n$, if and only if it satisfies the corresponding $\overline{G}$-invariant condition. Therefore, by Theorem \ref{thm:q>0} and Theorem \ref{thm:G-log}, there exists a $\overline{G}$-invariant convex body $K$ such that
	\[
	\mu = \widetilde{C}_{q}(K,  \cdot).
	\]
	Since $G\subset \overline G$, $K$ is also $G$-invariant.  Hence the sufficiency follows.
	
	 Conversely, suppose that there exists $K\in\KG$ such that
	$
	\mu=\widetilde C_q(K,\cdot).
	$
	By Lemma~7.2 of \cite{shan2026}, $K\in\mathcal K_{\overline G}$. If $0<q<n$, then Theorem~\ref{thm:G-invariant-subspace-mass}, applied to the closed group $\overline G$, together with the equivalence of $G$-invariant and $\overline G$-invariant subspaces above, implies that $\mu$ satisfies the $G$-invariant $q$-th subspace mass inequality. If $q=n$, then $\mu=V_K$, and the necessity follows from \cite{BH2016}, using the fact that $K\in\KG$ has centroid at the origin.
\end{proof}

\section{A structural consequence for $G$-invariant measures}

The $G$-invariant subspace concentration condition, and similarly the
$G$-invariant $q$-th subspace mass inequality, appears at first sight to be unable
to detect concentration phenomena on subspaces that are not preserved by $G$. Moreover, the family of $G$-invariant subspaces may be quite small, as illustrated
by irreducible groups and by the product groups described in \eqref{product group}.

There are two extreme cases in which the relation with the classical subspace
concentration condition is clear. If $G$ is irreducible, then
Proposition~\ref{irreducible-subspace-mass} shows that every $G$-invariant
measure automatically satisfies the classical subspace concentration condition,
that is, the condition tested on all subspaces. At the opposite extreme, if
$G=\{\pm I\}$, then every subspace is $G$-invariant, and the $G$-invariant
subspace concentration condition is exactly the classical subspace
concentration condition.

Outside these special cases, it is natural to ask whether a condition imposed
only on $G$-invariant subspaces can still control concentration on arbitrary
subspaces. More precisely, for $G$-invariant measures, does the $G$-invariant subspace
concentration condition already imply the classical subspace concentration
condition? The existence
theory for the $G$-invariant logarithmic Minkowski problem  established above gives an affirmative
answer and yields the following structural consequence for $G$-invariant
measures.

\begin{theorem}\label{thm:G-SCC-equivalence}
	Let $G\subset O(n)$ be a subgroup with no nonzero fixed points. Let $\mu$ be a nonzero finite $G$-invariant Borel measure on $\Sph$. Then the following are equivalent:

\smallskip

\noindent
\textup{(i)} The measure $\mu$ satisfies the $G$-invariant subspace concentration condition; namely, for every $G$-invariant subspace $L\subset \R^n$,
\[
\frac{\mu(L\cap \Sph)}{\mu(\Sph)}
\le \frac{\dim L}{n},
\]
and whenever equality holds for some $G$-invariant subspace $L\subset \R^n$, there exists a subspace $\widetilde L$ complementary to $L$ in $\R^n$ such that $\mu$ is concentrated on $S^{n-1}\cap (L\cup \widetilde L)$.

\smallskip

\noindent
\textup{(ii)} The measure $\mu$ satisfies the classical subspace concentration condition; namely, for every subspace $\xi\subset \R^n$,
\[
\frac{\mu(\xi\cap \Sph)}{\mu(\Sph)}
\le \frac{\dim \xi}{n},
\]
and whenever equality holds for some subspace $\xi\subset \R^n$, there exists a subspace $\xi'$ complementary to $\xi$ in $\R^n$ such that $\mu$ is concentrated on $S^{n-1}\cap (\xi\cup \xi')$.
\end{theorem}

\begin{proof}
	The implication \textup{(ii)}$\Rightarrow$\textup{(i)} is immediate.
	To prove \textup{(i)}$\Rightarrow$\textup{(ii)}, assume that \textup{(i)} holds. By Theorem~\ref{general exist}, there exists a $G$-invariant convex body $K$ such that
	\[
\mu=	V_K.
	\]
	Since $G$ has no nonzero fixed points, the centroid of $K$ is at the origin.
	By \cite{BH2016}, the cone-volume measure of every centered convex body satisfies the classical subspace concentration condition.  Since $\mu=V_K$, it follows that $\mu$ satisfies the classical subspace concentration condition, that is, \textup{(ii)}.
\end{proof}

This reveals a remarkable geometric structure of measures under group symmetry:
symmetry allows geometric concentration to be encoded by algebraic subspace data,
and the Minkowski problem serves as a bridge between algebra and geometry.

\begin{remark}\label{rem:sharp-constant}

We note that the factor $\dim(\cdot)/n$ in
Theorem~\ref{thm:G-SCC-equivalence} is sharp. Indeed, if this bound is replaced
by any strictly smaller one, then the equivalence in
Theorem~\ref{thm:G-SCC-equivalence} fails in general: one can no longer infer the
mass distribution on arbitrary subspaces solely from the information on
$G$-invariant subspaces. This can be seen from the following irreducible example.

Let $G\subset O(n)$ be the rotation symmetry group of the cube $[-1,1]^n$.
This group is irreducible, and hence the $G$-invariant subspace concentration
condition is vacuous.

Now consider the $G$-invariant probability measure
$
\mu=\frac{1}{2n}\sum_{i=1}^n(\delta_{e_i}+\delta_{-e_i})
$
on $\Sph$. For the one-dimensional subspace
$
\xi=\R e_1,
$
we have
$$
\mu(\xi\cap \Sph)=\mu(\{e_1,-e_1\})=\frac{1}{n}=\frac{\dim \xi}{n}.
$$
Therefore, if the bound $\dim(\cdot)/n$ were replaced by any strictly smaller
one, then the corresponding version of \textup{(i)} would still hold trivially for
this irreducible group $G$, whereas the corresponding version of \textup{(ii)}
would fail for the above measure $\mu$. Consequently, the equivalence in
Theorem~\ref{thm:G-SCC-equivalence} would no longer hold.
\end{remark}

\medskip
\noindent\textbf{Question.}
	Does a similar equivalence remain true under a larger upper bound?
	For instance, when $1<q<n$, does the $G$-invariant $q$-th subspace mass
	inequality imply some mass estimate on arbitrary subspaces?

\section{Appendix}

We collect a few elementary facts used in the paper.

\begin{lemma}\label{lem:fiber-identity}
	Let $K\in\mathcal K_o^n$ and let $L\subset\mathbb R^n$ be a  proper subspace.
	For $u\in\SL$, let
	\[
	\rho(u)=\rho_{K|L}(u),\qquad y_u=\rho(u)u,\qquad 
	A_u=K\cap(y_u+\lp).
	\]
Then
\begin{equation}\label{eq:fiber-decomposition}
	\begin{aligned}
		\{x\in K\setminus\{0\}:x/|x|\in\alpha_K^*(\SL)\}\cup\{0\}
		&=
		\bigcup_{u\in\SL}\operatorname{conv}\{0,A_u\} \\
		&=
		\bigcup_{u\in\SL}\bigcup_{0\le r\le \rho(u)}
		\frac r{\rho(u)}A_u .
	\end{aligned}
\end{equation}
Moreover, for every $u\in\SL$ and every $0<r\le \rho(u)$,
\begin{equation}\label{eq:fiber-section}
	\{x\in K\setminus\{0\}:x/|x|\in\alpha_K^*(\SL)\}\cap(ru+\lp)
	=
	\frac r{\rho(u)}A_u .
\end{equation}
\end{lemma}

\begin{proof}
We first prove the first identity in \eqref{eq:fiber-decomposition}.
 Fix $u\in\SL$ and $a\in A_u$. Then
	$P_La=y_u\in\partial(K|L)$. Choose $v\in\nu_{K|L}(y_u)\subset\SL$. For every
	$z\in K$, since $P_Lz\in K|L$, we have
	\[
	\langle z-a,v\rangle
	=
	\langle P_Lz-y_u,v\rangle
	\le 0 .
	\]
	Thus $v\in\nu_K(a) $, and hence $a/|a|\in\alpha_K^*(\SL)$. Therefore, using the convexity of $A_u$,
	\[
	\bigcup_{u\in\SL}\operatorname{conv}\{0,A_u\}
	\subset
	\{x\in K\setminus\{0\}:x/|x|\in\alpha_K^*(\SL)\}\cup\{0\}.
	\]
	
	Conversely, let $x\in K\setminus\{0\}$ satisfy
	$x/|x|\in\alpha_K^*(\SL)$. Set
	$
	z=\rho_K(x)x\in\partial K $.
	By the definition of $\alpha_K^*(\SL)$, there exists $v\in\nu_K(z)\cap L$.
	Since $v\in L$, for every $w\in K$ we have
	\[
	\langle P_Lw-P_Lz,v\rangle
	=
	\langle w-z,v\rangle
	\le 0 .
	\]
	Hence $v\in\nu_{K|L}(P_Lz)$, and therefore $P_Lz\in\partial(K|L)$. Thus, for some
	$u\in\SL$,
	\[
	P_Lz=\rho_{K|L}(u)u=y_u .
	\]
	It follows that $z\in A_u$. Since $x\in\operatorname{conv}\{0,z\}$, we get
	$x\in\operatorname{conv}\{0,A_u\}$. This proves
	\[
	\{x\in K\setminus\{0\}:x/|x|\in\alpha_K^*(\SL)\}\cup\{0\}
	=
	\bigcup_{u\in\SL}\operatorname{conv}\{0,A_u\}.
	\]
	Moreover, since each $A_u$ is convex,
	\[
	\operatorname{conv}\{0,A_u\}
	=
	\bigcup_{0\le r\le \rho(u)}\frac r{\rho(u)}A_u .
	\]
	Therefore \eqref{eq:fiber-decomposition} follows.

It remains to prove \eqref{eq:fiber-section}. Since $0<r\le\rho(u)$, 
\eqref{eq:fiber-decomposition} gives
$(r/\rho(u))A_u\subset\{x\in K\setminus\{0\}:x/|x|\in\alpha_K^*(\SL)\}$.
Moreover, since $A_u\subset y_u+\lp$ and $y_u=\rho(u)u$, we have
$(r/\rho(u))A_u\subset ru+\lp$.

 Conversely, let
\[
x\in \{z\in K\setminus\{0\}:z/|z|\in\alpha_K^*(\SL)\}\cap(ru+\lp).
\]
By \eqref{eq:fiber-decomposition}, there exist $u'\in\SL$, $a'\in A_{u'}$,
and $0<t\le 1$ such that $x=ta'$. Write $t=s/\rho(u')$ with
$0<s\le \rho(u')$. Since $a'\in A_{u'}$, we have $P_La'=\rho(u')u'$, and
therefore $P_Lx=su'$. On the other hand, $x\in ru+\lp$, so $P_Lx=ru$.
Since $r>0$ and $s>0$, it follows that $s=r$ and $u'=u$. Hence
\[
x=\frac r{\rho(u)}a'\in \frac r{\rho(u)}A_u .
\]
This proves \eqref{eq:fiber-section}.
\end{proof}

\begin{lemma}\label{lem:orthogonal-complement-invariant}
	Let $G\subset O(n)$, and let $L\subset\mathbb R^n$ be a $G$-invariant subspace.
	Then $L^\perp$ is also $G$-invariant.
\end{lemma}

\begin{proof}
	Let $g\in G$ and $x\in L^\perp$. For every $y\in L$, since $L$ is $G$-invariant, we have
	$g^{-1}y\in L$. Hence
	\[
	\langle gx,y\rangle=\langle x,g^{-1}y\rangle=0.
	\]
	Thus $gx\in L^\perp$, and so $gL^\perp\subset L^\perp$. Applying the same argument
	to $g^{-1}$ gives $g^{-1}L^\perp\subset L^\perp$, equivalently $L^\perp\subset gL^\perp$.
	Therefore $gL^\perp=L^\perp$ for every $g\in G$, and hence $L^\perp$ is
	$G$-invariant.
\end{proof}

\begin{lemma}\label{lem:projection-invariant}
	Let $G\subset O(n)$ and let $V\subset \mathbb R^n$ be a linear subspace. Then
	$V$ is $G$-invariant if and only if
	\[
	P_V g=gP_V
	\]
	for every $g\in G$.
\end{lemma}

\begin{proof}
Assume first that $V$ is $G$-invariant. By Lemma~\ref{lem:orthogonal-complement-invariant}, $V^\perp$ is also $G$-invariant. Let $x\in\mathbb R^n$. Write $x=P_Vx+(x-P_Vx)$, where $P_Vx\in V$ and
	$x-P_Vx\in V^\perp$. Then
	\[
	P_V(gx)=P_V(gP_Vx+g(x-P_Vx))=P_V(gP_Vx)=gP_Vx.
	\]
	Thus $P_Vg=gP_V$.
	
	Conversely, assume that $P_Vg=gP_V$ for every $g\in G$. If $v\in V$, then
	$P_Vv=v$, and hence
	\[
	P_V(gv)=gP_Vv=gv.
	\]
	Thus $gv\in V$. Therefore  $V$ is $G$-invariant.
\end{proof}


\end{document}